\documentclass[12pt,a4paper]{amsart}
\usepackage{amsmath,amsthm,amsfonts,relsize,units,amssymb}
\usepackage[colorlinks]{hyperref}

\DeclareMathOperator\ad{ad}
\DeclareMathOperator\Ad{Ad}
\DeclareMathOperator\R{\mathbb R}
\DeclareMathOperator\lie{Lie}

\DeclareMathOperator\ran{Ran}

\DeclareMathOperator\V{\mathsf{V}}
\DeclareMathOperator\der{{\texttt{Der}\,}(\V)}
\DeclareMathOperator\str{{\texttt{str}}(\V)}
\DeclareMathOperator\Str{{\texttt{Str}}(\V)}
\DeclareMathOperator\gl{{\mathsf{GL}}}
\DeclareMathOperator\glv{{\mathsf{GL}}(\V)}
\DeclareMathOperator\Aut{{\texttt{Aut}}(\V)}
\DeclareMathOperator\aut{{\texttt{aut}}(\V)}
\DeclareMathOperator\GO{{\mathsf{G}}(\Omega)}
\DeclareMathOperator\bv{{\mathsf{B}}(\V)}
\DeclareMathOperator\B{\mathsf{B}}
\DeclareMathOperator\h{\mathsf{H}}
\DeclareMathOperator\bh{\mathsf{B}(\h)}
\DeclareMathOperator\U{\mathrm{U}}

\begin{document}

\newtheorem*{theorem*}{Theorem}
\newtheorem{teo}{Theorem}[section]
\theoremstyle{definition}
\newtheorem{prop}[teo]{Proposition}
\newtheorem{lema}[teo]{Lemma}
\newtheorem{coro}[teo]{Corollary}
\newtheorem{defi}[teo]{Definition}
\newtheorem{rem}[teo]{Remark}
\newtheorem{ejem}[teo]{Example}
\newtheorem{problem}[teo]{Problem}

\markboth{}{}

\makeatletter

\title{\vspace*{-2cm}On the structure group of an infinite dimensional JB-algebra}
%22E65 Infinite-dimensional Lie groups and their Lie algebras (in Topological Groups and Lie groups)
%17C10 Structure theory for Jordan algebras
%58B20 Riemannian, Finsler and other geometric structures
%58B25 Group structures and generalizations on infinite-dimensional manifolds
%53C22 Geodesics (in Global differential Geometry)
%47B10 Operators belonging to operator ideals (in Operator theory)
%58D05 Groups of diffeomorphisms and homeomorphisms as manifold (in Global Analysis)
\date{}
\author{Gabriel Larotonda}
\address{Departamento de Matemat\'ica, FCEYN-UBA, and Instituto Argentino de Matem\'atica, CONICET, Buenos Aires, Argentina}
\email{glaroton@dm.uba.ar}
\author{Jos\'e Luna}
\address{Instituto Argentino de Matem\'atica ``Alberto P. Calder\'on'', CONICET, Buenos Aires, Argentina}
\email{jluna@dm.uba.ar}


\makeatother

\begin{abstract}{We extend several results for the structure group of a real Jordan algebra $\V$, to the setting of infinite dimensional JB-algebras. We prove that the structure group $\Str$, the cone preserving group $\GO$ and the automorphism group $\Aut$ of the algebra $\V$ are embedded Banach-Lie groups of $\glv$, and that each of the inclusions $\Aut\subset \GO\subset \Str$ are of embedded Banach-Lie subgroups. We give a full description of the components of $\Str$ via cones, isotopes and central projections. We apply these results to $\V=\bh_{sa}$ the special JB-algebra of self-adjoint operators on an infinite dimensional complex Hilbert space, describing the groups $\Str, \GO, \Aut$, their Banach-Lie algebras and their connected  components. We show that the action of the unitary group of $\h$ on $\Aut$ has smooth local cross sections, thus $\Aut$ is a smooth principal bundle over the unitary group, with structure group $S^1$.}
\end{abstract}

\maketitle

\setlength{\parindent}{0cm} %% para que no indente los parrafos nuevos

\tableofcontents

\section{Introduction}

The theory of real Jordan algebras $\V$ was introduced as a means to deal systematically with the observables in quantum mechanics by Jordan, Wigner and von Neumann \cite{koecher}, but from the very begininning there were difficulties in the setting of infinite dimensional algebras, and several well-known results for finite dimensional algebras are still lacking in the infinite dimensional setting. Most recently, the celebrated theorem of Koecher and Vinberg was extended to the setting of Banach Jordan algebras (JB-algebras for short) by Chu (see \cite{chu2} and the references therein), completing the characterization of the positive cone  $\Omega$ of $\V$ obtained by Kaup and Upmeier in \cite{kaup}.

\smallskip

The purpose of this paper is to extend well-known results of the structure group $\Str$ of a real Jordan algebra $\V$, to the setting of infinite dimensional JB-algebras: we prove that the structure group, the cone preserving group $\GO$ and the automorphism group $\Aut$ of the algebra $\V$ are embedded Banach-Lie groups of $\glv$, and that each of the inclusions $\Aut\subset \GO\subset \Str$ are of embedded Banach-Lie subgroups. We give a full description of the components of $\Str$ via cones and central projections, a result which generalizes naturally the presentation of $\Str$ for Euclidean (semi-simple, finite dimensional) Jordan algebras. In particular, this describes the isomorphic isotopes of the algebra $\V$. We apply these results to $\V=\bh_{sa}$ the special JB-algebra of self-adjoint operators on an infinite dimensional complex Hilbert space, describing the groups $\Str, \GO, \Aut$, their Lie algebras and their connected  components. With these results at hand, the Banach-Finsler geometry of these groups is studied in an accompanying paper \cite{larluna2}.

\smallskip

This paper is organized as follows: in Section \ref{s2} we go through the necessary definitions and properties of JB-algebras, presenting some infinite dimensional examples along the way, and finishing with a theorem characterizing the elements $x\in \V$ such that the spectrum of $U_x$ is positive. Section \ref{s3} contains the bulk of new results of the paper. We begin by reviewing the definition of the structure group $\Str$ of a JB-algebra $\V$, and rephrasing it in a way that allows us to present $\Str$ as an algebraic subgroup of $\glv$. Thus $\Str$ is an embedded Banach-Lie subgroup with the norm topology of $\bv$. Then we move on to the group $\GO$ preserving the cone, and again we show that it is a Banach-Lie subgroup of $\Str$, being its identity component. Further we show that the group $\Aut$ of automorphisms of the cone is an embedded Banach-Lie subgroup of $\GO$. Since the former is a strong deformation retract of the later, they have the same homotopy and in particular $\Aut$ and $\GO$ have the same number of connected components. We finish this section with a theorem that characterizes the components of $\Str$, which are copies of $\GO$, and each copy is uniquely determined by a central projection $p^2=p\in \V$. As an illustration of these results, in Section \ref{s4} we give a complete description of $\Str,\GO,\Aut$ for the special JB-algebra $\V$ of self-adjoint operators acting on a complex infinite dimensional Hilbert space $\h$. Further, we show that the smooth action $u\mapsto \Ad_u k$ of the unitary group $\mathcal U(\h)$ on $\Aut$ has smooth local cross sections, inducing a principal $S^1$-fiber bundle $\pi:\Aut\to \mathcal U(\h)$.

\section{Jordan algebras, cones and the spectrum}\label{s2}

In this section we survey the main objects and tools of theory of Jordan Banach algebras, we present some relevant examples, and we finish the section with a characterization of the positive cone $\Omega$ of a Jordan Banach algebra $\V$, in terms of the quadratic representation (this is Theorem \ref{upositive}, which is a well-known result for finite dimensional algebras). Throughout, if $V$ denotes a real Banach space, we will denote with $V^*$ the topological dual of $\V$,  with $\B(V)$ the algebra of bounded linear operators in $V$, and with $\gl(V)$ the group of invertible operators in $V$. 

\subsection{Cones in Banach Spaces}

\begin{defi}[Cones] 	A nonempty set $\Omega\subset\V$ is a \textit{convex cone} if it satisfies that $\Omega + \Omega \subset \Omega$ and $\lambda \Omega \subset \Omega$ for all positive $\lambda$. The cone is \textit{proper} if $\Omega \cap -\Omega = \{0\}$. The cone is \textit{reproducing} if $\V=\Omega-\Omega$. Every cone is a convex set, and every proper cone induces a partial order: $x \leq y$ if $y-x \in \Omega$. If $\Omega$ is a cone, then its closure $\overline{\Omega}$ is also a cone. If $\Omega$ is an open cone then  $\overline{\Omega}^{\circ} = \Omega$ (see \cite[Lemma 2.2]{chu} for the proof). Then it is plain that $\overline{\Omega} = \{v \in \V: v \geq 0\}$, so we can recover the cone from the partial order.
\end{defi}

\begin{defi}[Order units]\label{ordernorm}
	Let $\V$ be a real Banach space, and $\Omega\subset \V$ a cone. An element $e$ is called an order unit if for every $x$ in $\V$ there exists a positive $\lambda\in\R$ such that $-\lambda e \leq x \leq \lambda e$. An order unit $e$ is called \textit{archimedean} if all $x$ in $\V$ satisfy that if $\lambda x \leq e$ for every positive $\lambda$, then $x \leq 0$. 	An archimedean order unit $e$ induces a norm $||\cdot||$ on $\V$, by means of $\;\|x\|_e = \inf\{\lambda > 0 : -\lambda e \leq x \leq \lambda e\}$. 
	
The space $(\V,e)$ is a \textit{complete archimedean order unit space} if the order unit norm $||\cdot||$ is complete.  If $\V$ has a order unit then $\V$ it is reproducing: take $x$ in $\V$, then there exists positive $\lambda$ such that $-\lambda e \leq x \leq \lambda e$. From this it follows that both $x_1 =\frac{\lambda e +x}{2}$ and $x_2 = \frac{\lambda e - x}{2}$ are positive, and $x = x_1-x_2$.
\end{defi}

\begin{defi}[Positive maps]
A linear map $T:(\V,e) \to (W,u)$ is \textit{positive} if it maps the cone on $\V$ into the cone on $W$. It is called a \textit{positive linear functional} if $(W,u)$ is the real numbers with archimedean order unit 1. The set of positive functionals will be denoted by $\Omega^*$.
\end{defi}

\begin{defi}[Normal cone]
	A cone $\Omega$ is \textit{normal} if there exists $\delta >0$ such that for every $u$, $v$ in $\Omega$ with $||u|| = ||v|| =1$, $||u + v|| \geq \delta$. Every closed normal cone is proper.
\end{defi}

\begin{lema}\label{normal}
		Let $\V$ be a real Banach space with a cone $\Omega$. The following are equivalent:
\begin{enumerate}
\item $\Omega$ is a normal cone 
\item $||x|| \leq M ||e||\, ||x||_e$ for some constant $M$ independent of $x,e$
\item the norm is semi-monotone: $\exists K>0$ s.t. $0 \le x\le y$ implies $\|x\|\le K\|y\|$
\item $\Omega^*$ is a reproducing cone.
\end{enumerate}
	\end{lema}

\begin{proof}
See \cite[Theorems 1.1 and 1.2]{kras} for the equivalence or normality with the second and third assertions, and \cite[Chapter 5, Section 3]{schaefer} for the equivalence with the fourth.
\end{proof}

\begin{defi}[Symmetric cone]
	Let $\V$ be a real Banach space. We say that a proper open cone $\Omega\subset\V$ is  \textit{symmetric} if it is
	\begin{enumerate}
		\item (self dual) $\Omega = \{x \in \V: \varphi(x) > 0 \text{ for every }\varphi \in \Omega^*\}$ and
		\item (homogeneous) for every $x$, $y$ in $\Omega$ there exists an isomorphism $g: \V \to \V$ such that $g(x) = y$.
	\end{enumerate}
\end{defi}

\begin{rem}Note that a symmetric cone is also symmetric in the following sense: if $\Omega$ is a symmetric cone in a Banach space $\V$,   then
	\begin{equation*}
		\overline{\Omega} = \{x \in \V: \varphi(x) \geq 0 \text{ for every }\varphi \in \Omega^*\}.
	\end{equation*}
This is because if we take $x$ in $\V$ such that $\varphi(x) \geq 0$ for every $\varphi \in \Omega^*$, let $e\in\Omega$. Then for every $n$ we have $\varphi(x+ \nicefrac{1}{n}\,e) =\varphi(x) + \nicefrac{1}{n}\,\varphi(e) >0$, hence $x\in \overline{\Omega}$.
\end{rem}

Order units are key in ordered spaces, as they provide a norm. It is possible to characterize cones with order units, and compare the order unit norm  with the original norm. We state the precise result below, see \cite[Lemma 2.5]{chu} for a proof:

\begin{lema}\label{ref:hayUnidOrd}
	Let $\V$ be a real vector space with norm $||\cdot||$ and let $\Omega\subset \V$ be an open proper cone. Give $\V$ the order induced by $\overline{\Omega}$. Then every element $e$ in $\Omega$ is an order unit, and if $e$ is archimedean then $||\cdot||_e \leq c||\cdot||$ for some positive $c$.
\end{lema}

This tells us that in a symmetric cone we have order units. We will now discuss a couple of examples, which shows the interplay of these properties and their limitations:

\begin{ejem}[A self-dual reproducing cone with empty interior, which is not contained in any proper subsapce]\label{ejem} Let $\V=\ell^2(\mathbb N)$, let $C=\{x\in \V: x_n> 0\}$. Then it is plain that $C$ is a self-dual cone, in particular convex. We claim that $C^o=\emptyset$, that $C$ is reproducing and that $C$ is not contained in a proper subspace. For the first claim, let $x\in C$, let $r>0$ and take $n$ such that $x_n<r/2$ and let $y$ be obtained from $x$ by replacing the $n$-th entry of $x$ by $y_n=x_n-r/2<0$. Then $y\in B_r(x)$ but $y\notin C$, thus $C$ has empty interior. Now let $z\in \V$, decompose $z=z^+-z^-$ where $z^+$ consists of the positive entries of $z$ (zero elsewhere) and $z^-$ the negative entries of $z$. Now, $z^+$ and $z^-$ do not belong in $C$, as some of their entries can be null. But consider $\tilde{z}^+ = z^+ + (\frac{1}{n})_{n \in \mathbb{N}}$ and $\tilde{z}^- = z^- + (\frac{1}{n})_{n \in \mathbb{N}}$; both these elements belong in $C$ and it is plain that $z = \tilde{z}^+ - \tilde{z}^-$ thus $\V=C-C$. Finally, assumme that $C$ is contained in an affine subspace $C\subset x_0+W$; then, $C-x_0 \subset W$. But $\V = C - C = (C-x_0) - (C-x_0) \subset W - W$, which shows that $W=\V$.
\end{ejem}

\begin{ejem}[A self-dual open cone that is not normal]\label{c1ex} Let $\V=C^1[0,1]$ with the usual norm $||f||_{C^1} = ||f||_{\infty} + ||f'||_{\infty}$, let $C = \{f \in V: f(x)> 0\text{ for all }x\}$. It is plain that $C$ is open, as $||\cdot||_{\infty} \le ||\cdot||_{C^1}$. But $C$ is not normal: it is obvious that $0 \le x^n \le x$ for all $n$, but $||x^n|| = n+1$ and $||x|| = 2$. On the other hand, $C$ is self-dual: if $L \in (C^1[0,1])^*$, then $L(f) = \int_0^1 f d\mu_1 + \int_0^1 f' d\mu_2$ for $\mu_1$, $\mu_2$ signed borel measures on $[0,1]$. We claim that $C^* = \{L : L(f) = \int_0^1 f d\mu \text{ for positive measure }\mu\}$. One inclusion is obvious. Now, let $L \in C^*$. $L(f) = \int_0^1 f d\mu_1 + \int_0^1 f' d\mu_2$ for $\mu_1$, $\mu_2$ signed borel measures. Let $P_1$, $N_1$, $P_2$ and $N_2$ be the positive and negative sets of $\mu_1$ and $\mu_2$ respectively. Suppose $N_1$ is not empty. Take $f$ positive differentiable function such that is null in $P_1$, increasing in $N_1 \cap N_2$ and decreasing in $N_1 \cap P_2$. Then $L(f) \le 0$, which is absurd. Then $N_1$ is empty and $\mu_1$ is positive. Analogously, $P_2$ and $N_2$ are empty and $\mu_2$ is the null measure. Now, $C$ is self-dual as if $\int_0^1fd\mu > 0$ for every positive measure $\mu$, then $f$ is a positive function. Lets compute the order norm of $C$. We have
$$
||f||_e = \inf \{\lambda>0 : -\lambda < f(x) < \lambda \text{ for all }x\} = \sup{|f(x)|} = ||f||_{\infty}.
$$
With this norm, $C^1[0,1]$ is not a Banach space, but its closure $C[0,1]$ is.
\end{ejem}

In many cones, the original norm is actually equivalent to the order unit norm:

\begin{lema}\label{NormalEqOrden}
	Let $V$ be a real Banach space with a symmetric cone $\Omega\subset V$. Give $V$ the partial order induced by $\overline{\Omega}$, and let $e$ be an element in $\Omega$. Then $e$ is an archimedean order unit, and the unit norm $||\cdot||_e$ is equivalent to the original norm $||\cdot||$ if and only if the cone $\Omega$ is normal.
\end{lema}
\begin{proof}
Take $x$ in $V$ such that $\lambda x \leq e$ for every positive $\lambda$. Then, as $e-\lambda x$ is positive, we have that $\varphi(e-\lambda x)\geq 0$ for every positive functional $\varphi$. Then, $\varphi(e) \geq \lambda \varphi(x)$. 	As $\varphi(e)>0$, we have that it is an order unit of the real numbers, which is archimidean. Then $\varphi(x) \leq 0$. As this is true for every positive functional, we have that $x$ is negative; thus $e$ is archimedean. By the Lemma \ref{ref:hayUnidOrd}, $||\cdot||_e \leq c||\cdot||$ for some positive $c$.  From Lemma \ref{normal} we see that we have a reversed inequality if and only if $\Omega$ is normal.
\end{proof}

\subsection{Jordan Algebras and the spectrum}

Let $\V$ be a real vector space with product $\circ$, possibly infinite dimensional. Then $(\V,\circ)$ is a \textit{Jordan algebra} if $\circ$ is commutative and
\begin{equation}\label{eq:jordanAlg}
	x^2 \circ (x \circ y) = x\circ (x^2 \circ y).
\end{equation}

Every associative algebra can be made into a Jordan algebra with the Jordan product $x \circ y = \frac{xy + yx}{2}$. These algebras are the \textit{special} Jordan algebras. A fundamental theorem in the theory follows (see \cite[pag. 41]{jacobson}):

\begin{teo}[McDonald's Theorem] Every polynomial Jordan identity in three variables and $1$, which is of degree at most $1$ in one of these variables and holds for all special Jordan algebras, is valid for all Jordan algebras.
\end{teo}

\begin{defi}[Quadratic representation]For fixed $x\in \V$, define the operator $L_x: \V \to \V$ by means of $L_xy = x\circ y$, and consider the linear operator $\U_x = 2L_x^2 - L_{x^2}$. This is a representation of $\V$ in $\bv$ which is quadratic in $x$, the \textit{quadratic representation} of $\V$. If we compute the quadratic representation in an associative algebra with the Jordan product $a\circ b=1/2(ab+ba)$, we obtain $\U_xy = xyx$.

The quadratic representation gives us a bilinear representation:
\begin{equation}\label{du}
	\U_{x,y} = \frac{1}{2}(\U_{x+y} - \U_x - \U_y) = \frac{1}{2} D_y\U(x) = L_xL_y + L_yL_x - L_{x\circ y}.
\end{equation}
where $D_y f(x)$ denotes the differential of the map $f$ at the point $y$, in the direction of $x$. In an associative algebra with the Jordan product $a\circ b=1/2(ab+ba)$, we obtain $\U_{x,y}z = \frac{1}{2}(xzy+yzx)$.

%Define the $V$-operator as $V_{x,y}(z) = \U_{x,z}(y)$. It is obvious that the $V$-operators are bilinear.

It is possible to build  the Jordan structure around the quadratic representation, to get an equivalent formulation. The most important property of this representation is the \textit{fundamental formula}:
\begin{equation}\label{eq:fundamf}
	\U_{\;\U_x y} = \U_x\U_y\U_x = \U_{\;\U_x}(\U_y),
\end{equation}
where the last $\U$ is the quadratic representation of the associative algebra of operators $\bv$ with the Jordan product $a\circ b=1/2(ab+ba)$. A proof for JB-algebras can be derived using that it holds in a special JB algebra, and then using the theorem of McDonald cited above.
\end{defi}

\begin{rem}[Invertible elements]\label{invert}
	An element $x\in \V$ is \textit{invertible} if there exists $y\in \V$ such that $\U_xy = x$ and $\U_xy^2 = 1$. Note that althoug this definition implies that $x\circ y = 1$, it is not equivalent to it.  However: an element $x$ is invertible if and only if $\U_x$ is an invertible operator (for a proof see \cite[Part II, Criterion 6.1.2]{mccrimmon}). Moreover, $x,y$ are invertible if and only if $\U_xy$ is also invertible: this is plain from  $\U_{\U_xy}=\U_x\U_y\U_x$ (the fundamental formula) and the previous result.
\end{rem}

\begin{defi}[Spectrum]
	For $x\in \V$, the \textit{spectrum} of $x$ is defined as
	\begin{equation*}
		\sigma(x) = \{\lambda \in \R \text{ such that } x-\lambda 1 \text{ is not invertible}\}.
	\end{equation*}
Being the same set as the spectrum of $a$ as an element of the $C^*$-algebra generated by $1,a$, the spectrum is a nonempty compact set with nice properties, see Remark \ref{expsq} below, see \cite[p. 18]{alfschul} and \cite[$\S$ 10]{bonsall} for further details.
\end{defi}

\begin{defi}[Positive cone of a Jordan algebra]	Given $\V$ a Jordan algebra an element $x\in \V$ is \textit{positive} if $\sigma(x)$ is contained in the positive numbers. We denote by $\Omega$ the set of positive elements in $\V$. 	The set $\Omega$ is a convex cone and $\Omega \cap (-\Omega)=\emptyset$, hence $\Omega\subset \V$ is an open proper cone. Thus we have partial order in $\V$ as in the previous section:  given $x,y\in \V$, then $x<y$ if $y-x\in \Omega$.  It is plain that  $\overline{\Omega}$ is the set of elements with nonnegative spectrum, by the lower-semicontinuity of the spectrum map.
\end{defi}

\subsubsection{JB-algebras}

Our general reference for JB-algebras are the books \cite{stormer} and \cite{bonsall}.
\begin{defi}[JB-algebras]
	Let $\V$ be a Jordan algebra with norm $\|\cdot\|$ such that $(\V,\|\cdot\|)$ is a Banach space. The space $\V$ is a \textit{JB-algebra} if
$$
1)\quad\|x\circ y\|\leq \|x\| \|y\|\quad\qquad\qquad 2)\quad \|x^2\| = \|x\|^2\quad \qquad\qquad 3)\quad \|x^2\| \leq \|x^2 + y^2\|.
$$
\end{defi}

Then  every JB-algebra is \textit{formally real}: if $x^2 + y^2 = 0$, then both $x=y=0$.  The \textit{order norm} in $\V$ is the norm induced by the partial order and the chosen unit order $e=1$ in the cone $\Omega\subset \V$ (Definition \ref{ordernorm}):
	\begin{equation*}
		\|x\| = \inf\{\lambda > 0 : -\lambda 1 < x < \lambda 1\}.
	\end{equation*}

\begin{rem}
Every JB-algebra is archimedean: for every $x$ in $\V$ there exists $\lambda > 0$ such that $-\lambda 1< x < \lambda 1$ (see \cite[Proposition 3.3.10]{stormer} for a proof). Moreover, the order norm coincides with the original norm of the space $\V$ \cite[Proposition 3.3.10]{stormer}.
\end{rem}

As we said before, if $x$ has inverse $y$ it implies that $x \circ y = 1$ but this is not equivalent to being invertible. In other words, $L_x$ invertible implies that $x$ is invertible, but not the other way around. From this it is apparent that in general $\sigma(x)\subset \sigma(L_x)$ without equality.

\medskip

\begin{defi}
A subalgebra $\mathcal A\subset \V$ is called a \textit{strongly associative subalgebra} if $[L_a,L_b]=0$ for all $a,b\in \mathcal A$. For $x\in \V$, we have $[L_{x^n},L_{x^j}]=0$ for all $n,j$; thus if we let $\mathcal{C}(x)$ denote the closed subalgebra generated by $x$, then $\mathcal C(x)$ is a strongly associative subalgebra. 
\end{defi}

Thus there is an alternate characterization of invertibility, we include the proof to illustrate the methods:
\begin{lema}
	An element $x\in \V$ is invertible if and only if there exists $y\in\mathcal{C}(x)$ such that $x \circ y = 1$.
\end{lema}
\begin{proof}
	If $x$ is invertible, then $\U_x$ is an invertible operator and $x^{-1} = \U_x^{-1}x$. Since the inverse of $\U_x$ can be approximated with polynomials in $\U_x=2(L_x)^2-L_x^2$, and  $L_x,L_{x^2}$ commute and are self-maps of $\mathcal{C}(x)$, we see that $x^{-1} = \U_x^{-1}x\in \mathcal{C}(x)$. Conversely, if there exists $y\in \mathcal{C}(x)$ such that $x \circ y = 1$ then $y$ also satisfies the condition $\U_xy^2 = 1$, as $\mathcal{C}(x)$ is associative.
\end{proof}

\begin{rem}[Square roots and logarithms]\label{expsq}
The order norm on a JB-algebra allows the continuous functional calculus on $\V$. The associative commutative Banach algebra $\mathcal{C}(x)$ is isometrically isomorphic to $C(\sigma(x))$; a proof can be found in \cite[Theorem 3.2.4]{stormer}, and it uses a complexification of $\mathcal{C}(x)$ and the known spectral theorem for complex algebras. Then $\Omega$ can be characterized as the cone of squares of $\V$: every positive element has an square root, and as $\sigma(x^2) = \{\lambda^2, \lambda \in \sigma(x)\}$, every square is positive. Likewise, we can also characterize $\Omega=e^{\V}$: every positive element has a real logarithm, and  $\sigma(e^x) = \{e^\lambda, \lambda \in \sigma(x)\}$, so the exponential of every element is positive. Note that the quadratic representation is injective in the cone $\Omega$. Take $x$, $y$ positive elements such that $\U_x = \U_y$. Then $x^2 = \U_x(1) = \U_y(1) = y^2$. Since  there exists an unique positive square root, we have $x=y$.
\end{rem}

The following formula relating the product and the quadratic representation is well-known; it will  be useful later and we include a proof to shows that is also works well in the complexification of $\V$ (see the next section):
\begin{lema}\label{expu}
	Let $v$ be an element in $\V$. Then $e^{2L_v} = \U_{e^v}$.
\end{lema}
\begin{proof}
	Let $F: \R \to \bv$ be $F_t = \U_{e^{tv}}$. As both $e^{\frac{t}{2}v}$ and $e^{sv}$ belong to $\mathcal{C}(v)$, we have that $\U_{e^{\frac{t}{2}v}}e^{sv}= e^{(s+t)v}$. Then,
	\begin{equation*}
		F_{s+t} = \U_{(s+t)v} = \U_{\U_{\exp({\frac{t}{2}v})}e^{sv}} = \U_{e^{\frac{t}{2}v}}\U_{e^{sv}}\U_{e^{\frac{t}{2}v}} = F_{\frac{t}{2}}^2F_s.
	\end{equation*}
	Replacing with $s=0$, we have that $F_t = F_{t/2}^2$. Then $F_{s+t} =F_s F_t=F_tF_s$ thus $F$ is a one-parameter group.  Now $		F'(0) = D_1(\U_-)e^{0v}v = 2\U_{v,1} = 2L_v$,  so it must be that that $F(t) = e^{2tL_v}$.
\end{proof}

\begin{rem}[The positive cone of a JB-algebra is symmetric and normal] The cone $\Omega$ is obviously proper and open, and it is homogeneous, as for every positive $x$ and $y$ we have that $y = \U_{y^{1/2}} \U_{x^{-1/2}}x$. The cone is self dual: this follows from the fact that $x\ge 0$ if and only $\varphi(x)\ge 0$ for any $\varphi\in \Omega^*$, and the fact that for any $x\in \V$ there exists $\varphi\in \Omega^*$ such that $\varphi(x)=\|x\|$ (see \cite[Lemma 1.2.5]{stormer} for a proof). Finally, the cone $\Omega$ is also normal  by Lemma \ref{NormalEqOrden} since in this case the original norm is the order unit norm.
\end{rem}

A Banach space with Jordan structure and symmetric cone of positives $\Omega$ cannot be a JB-algebra if the cone $\Omega$ is not normal, by Lemma \ref{NormalEqOrden}. Let us show an explicit example of this situation:

\begin{rem}[A Banach Jordan algebra whose norm is not equivalent to a JB-algebra norm] 	As we have seen in Example \ref{c1ex}, the cone $\Omega = \{f: f(x) >0 \text{ for all } x\}$ in $C^{1}[0,1]$ with the norm $||f||_{C^1} = ||f||_{\infty} + ||f'||_{\infty}$ is not normal. Note that $(C^{1}[0,1],||\cdot||_{C^1})$ is a Banach space and that $\Omega$ is the cone of positive elements, as $\sigma(f) = Im(f)$. As the cone is not normal, $||\cdot||_{C^1}$ cannot be equivalent to a JB-algebra norm. %As we have also seen in the same remark, $C[0,1]$ is a complete order unit space, as the supremum norm is the order norm. It is easy to see that it is a JB-algebra with the supremum norm
\end{rem}

The following nice characterization was recently proved by Chu \cite[Theorem 3.2]{chu2}, and gives a good geometric picture of the correspondence among JB-algebras and homogeneous normal cones:

\begin{teo} Let $\Omega$ be an homogeneous normal cone in a real Banach space $\V$. Then $\V$ is a unital JB-algebra (in an equivalent norm) with $\Omega$ as its positive cone, if and only if there exists a symmetric Banach manifold structure in $\Omega$ in the sense of Loos.
\end{teo}

See \cite[Example 1.6]{chu} for the axiomatic definition of a structure of symmetric Banach manifold $X$, which essentially involves that each point of $x\in X$ is the unique fixed point of an involutive symmetry around $x$.

\subsection{Spectrum of the representations and positivity}

Functional calculus will link the positivity of elements with the positivity of their quadratic operators. 

\smallskip

Let $\V$ a JB-algebra, and consider  $\V^{\mathbb C}=\V\oplus i \V$ the Jordan algebra complexification of $\V$. We have that $\V^{\mathbb C}$ is a JB$^*$-algebra with a natural involuton $(a+ib)^*=a-ib$, and $\V=\{v\in \V^{\mathbb C}: v^*=v\}$.

\begin{defi}[Numerical ranges]	Let $x\in \V^{\mathbb C}$. The \textit{numerical range} of $x$ is the set
	\begin{equation*}
		V(x) = \{\phi(x): \phi\in (\V^{\mathbb C})^*, \phi(1) = \|\phi\| = 1\}\subset\mathbb C.
	\end{equation*}
An element $x\in \V^{\mathbb C}$ is \textit{Hermitian} if $V(x)\subset\mathbb \R$. We have that $\V=\textrm{Herm}(\V^{\mathbb C})$ by Theorem 7 in \cite{youngson}. The set $V(x)$ is compact, convex and nonempty. If $x\in \V$, we have that $V(x)=co(\sigma(x))$, where $co(\Sigma)$ is the convex hull of the set $\Sigma\subset\mathbb C$ (see \cite[$\S$ 10]{bonsall} for a proof of these facts). 

\smallskip

Let $X$ be a complex Banach space, let $T\in \B(X)$. The \textit{intrisic numerical range} of $T$ is the set
$$
V(T)=\{\psi(T): \psi\in \B(X)^*, \|\psi\|=\psi(1)=1\},
$$
and the \textit{spatial numerical range} of $T$ is the set
	\begin{equation*}
		W(T) = \{\psi(Tz): \psi\in X^*,z\in X, \psi(z) = 1, \|\psi\| = \|z\| = 1\}.
	\end{equation*}
We have $\overline{co(W(T))}=V(T)$ and $V(T)$ is compact nonempty and convex (see \cite{martin}). 	An operator $T\in \B(X)$ is \textit{Hermitian} if $W(T)\subset\mathbb R$ (equivalently, if $V(T)\subset\mathbb R$). 
\end{defi}	

\begin{lema}
For $x\in \V^{\mathbb C}$, consider the complexification $\mathbb L_x$ of $L_x$ given by $\mathbb L_x(v+iw)=L_xv+iL_xw$ for $v,w\in \V$. Then $V(x) = W(\mathbb L_x)$. 
\end{lema}
\begin{proof}Suppose there is a value $\phi(x)$ in $V(x)$ with $\phi(1) = \|\phi\| = 1$. Take $z=1\in \V$, take $\psi = \phi$; then $\psi(z) = \phi(1) = 1$ and $\|z\| =\|1\| =1$, $\|\psi\| = \|\phi\| = 1$, so $\psi(\mathbb L_xz)$ belongs to $W(\mathbb L_x)$. But $\psi(\mathbb L_xz) = \phi(x)$, so we have $V(x) \subset W(\mathbb L_x)$. Reciprocally, suppose there is a value $\psi(\mathbb L_xz)$ in $W(\mathbb L_x)$ with $\psi(z) = 1$, $\|\psi\| = \|z\| = 1$. Define $\phi(y) = \psi(\mathbb L_yz)$ which is clearly linear and bounded. We have that $\phi(1) = \psi(z) = 1$. Moreover,
	\begin{equation*}
\|\phi\| = \|\psi \circ \mathbb L_z\| \leq \|\psi\| \,\|\mathbb L_z\| = \|\psi\|\,\|z\| = 1,
	\end{equation*}
	and as $\phi(1)=1$, $\|\phi\| = 1$. Then $\phi(x)$ belongs to $V(x)$. But $\phi(x) = \psi(\mathbb L_xz)$, so we have $W(\mathbb L_x) \subset V(x)$.
\end{proof}

\begin{rem}[$x$ versus $L_x$]\label{xvslx} If $x\in \V$, we consider the complexification $\mathbb L_x\in \B(\V^{\mathbb C})$ and by the previous lemma and remarks we have
$$
V(\mathbb L_x)=\overline{co(W(\mathbb L_x)}=\overline{co(V(x))}=V(x)\subset\R. 
$$
Thus $\mathbb L_x$ is Hermitian and $V(\mathbb L_x)=co(\sigma(\mathbb L_x))$ (see \cite[$\S$ 10]{bonsall}). Hence for $x\in \V$ we have
$$
co(\sigma(\mathbb L_x)) = V(\mathbb L_x) = \overline{co(W(\mathbb L_x))} = \overline{co(V(x))} = V(x) = co(\sigma(x))\subset\mathbb R
$$
(this is essentially Theorem 6 in \cite{youngson2}).
\end{rem}

\begin{defi}If we say that $\mathbb L_x$ is \textit{positive} if $\sigma(\mathbb L_x)\subset (0,+\infty)$, and that it is \textit{negative} if $\sigma(\mathbb L_x)\subset (-\infty,0)$, we see that $x\in \V$ is non-negative if and only if $\mathbb L_x$ is non-negative; and $x$ is non-positive if and only if $\mathbb L_x$ is non-positive.  Now it is not hard to see that $\sigma(\mathbb L_x)\cap\mathbb R= \sigma(L_x)$, where the latter is the set $\{t\in\mathbb R: L_x-t1 \textrm{ is not invertible in } \bv$\}. Thus we can conclude that \textit{$L_x$ has non-negative spectrum (resp. non-positive) if and only if $x\in \overline{\Omega}$ (resp. $x\in -\overline{\Omega}$)}.
\end{defi}

Recall that an element $x\in \V$ is \textit{central} if $L_x$ commutes with $L_z$ for all $z\in\V$.

\begin{defi}[Central symmetries]\label{cpcs} Let $p=p^2\in\V$ be an idempotent. We say that $p$ is a \textit{central projection} if $p$ is a central idempotent of $\V$. The element $\varepsilon_p=2p-1$ is a central \textit{symmetry}, that is $\varepsilon_p$ is invertible and $\varepsilon_p=\varepsilon_p^{-1}$. The elements $0,1\in \V$ are central projections with $\varepsilon_1 = 1$, $\varepsilon_0 = -1$. If $\varepsilon$ is a central symmetry then $\U_{\varepsilon}=1$, inded: 
$$\U_{\varepsilon}z = 2 \varepsilon\circ (\varepsilon\circ z) - \varepsilon^2\circ z = 2 z \circ \varepsilon^2 - z = z.
$$
\end{defi}

\begin{rem}[$\U$ is usually not Hermitian but has real spectrum]\label{unoth} Unlike the $L$ operators, the quadratic operator $\U_x$ is Hermitian if and only if $x$ is central (see \cite[Theorem 14]{youngson2}). Now, if $\mathbb U_x$ denotes the complexification of $\U_x$ to $\V^{\mathbb C}$, with the same proof than Lemma \ref{expu} we have that $e^{2 \mathbb L_x}=\mathbb U_x$ for any $x\in \V^{\mathbb C}$. In particular this shows that for $Z=\bv$, the inclusion $\exp(\textrm{Herm}(Z))\subset \textrm{Herm}(Z)$ does not hold, answering in the negative Problem 6.1(b) in \cite{neeb}; if $x\in\V$ is not a central element, then we have an Hermitian element $a=\mathbb L_x\in \B(Z)$ such that $a^2=1/2(\mathbb U_x+\mathbb L_{x^2})$ is not Hermitian.
\end{rem}

\begin{rem}
If $x\in \Omega$ we can write $x=e^v$ for $v\in \V$, and then $\sigma(\mathbb U_x)=\sigma(e^{2\,\mathbb L_v})\subset (0,+\infty)$,
since $\sigma(\mathbb L_v)\subset \mathbb R$. Then $\sigma(\U_x)=\sigma(\mathbb U_x)\cap \mathbb R\subset (0,+\infty)$ also.
\end{rem}

What follows is the main result of this section. For simple Euclidean Jordan algebras (which are finite dimensional), one can use the Pierce decomposition relative to a Jordan frame to obtain a very simple proof of this theorem (see \cite[Lemma VIII.2.7]{faraut}).

\begin{teo}\label{upositive}
Let $x\in \V$ a JB-algebra. Then $\sigma(\U_x)\subset (0,+\infty)$ if and only if $x = v \varepsilon$ where $v$ is positive and $\varepsilon$ is a central symmetry of $\V$.
\end{teo}
\begin{proof}
Assume that $x = v \varepsilon$ with positive $v$ and central symmetry $\varepsilon$. We have that $\U_x = \U_{v\varepsilon} = \U_v\U_{\varepsilon}=\U_v$, as $\varepsilon$ is central; so we can assume that $x$ is positive. But then $\sigma(\U_x)\subset (0,+\infty)$ follows from the previous remark.  To prove the converse statement we will use the Pierce decomposition. Given two supplementary orthogonal idempotents $e$ and $e' = 1-e$ one can define three projections $\U_e$, $\U_{e'}$ and $2\U_{e,e'}$. They form a supplementary family of projection operators on $\V$, which breaks into the direct sum of their ranges. For more details see \cite[Part II, Chapter 8]{mccrimmon}. Let $x\in \V$ an invertible element,  let $\chi_+$ denote  the characteristic function of the positive numbers, and let $\chi_-$ denote  the characteristic function of the negative numbers (both continuous functions on $\sigma(x)$, as $0$ does not belong to the spectrum of $x$). Let $p_+ = \chi_+(x)$ and $p_- = \chi_-(x)$. Since $\chi_+ + \chi_- = 1$, $\chi_+ \chi_- = 0$ and $\chi_{\pm}^2 = \chi_{\pm}$, $p_+$ and $p_-$ are supplementary orthogonal idempotents, ie. $p_{\pm}^2 = p_{\pm}$, $p_+ + p_- = 1$ and $p_+\circ p_- = 0$. Note that $x$ is either positive or negative if and only if they are the unit and the null elements. Consider the operators $\U_{p_+}$, $\U_{p_-}$ and $2\U_{p_+,p_-}$. Since $p_+$ and $p_-$ are idempotents belonging to the strongly associative subalgebra $\mathcal C(x)$, by \cite[Chapter 1, Section 8, Lemma 8]{jacobson} we have that $\U_{p_+}^2=\U_{p_+^2} = \U_{p_+}$, and $\U_{p_-}^2=U_{p_-^2} = \U_{p_-}$. So $\U_{p_+}$ and $\U_{p_-}$ are projections; moreover by the same lemma, $\U_{p+}\U_{p-}=\U_{p_+\circ p_-}=\U_0=0$, thus they are orthogonal and their sum $S=\U_{p+}+\U_{p-}$ is also a projection. As $\U_{p_+}+\U_{p_-} + 2\U_{p_+,p_-} = \U_{p_++p_-} = \U_1 = Id$, from $2\U_{p_+,p_-}=1-S$ we conclude that $2\U_{p_+,p_-}$ is also a projection. Let $J_0$ be the range of $2\U_{p_+,p_-}$, $J_+$ the range of $\U_{p_+}$ and $J_-$ the range of $\U_{p_-}$. Note that this decomposition is trivial if $x$ is positive: in this case $J_+ = \V$ and the other two spaces are the null space. Similarly, if $x$ is negative, $J_-= \V$ and the other two spaces are null. In general, $J_0$ is trivial if and only if $p_+$ is central. If $p_+$ is central so is $p_-$ and $L_{p_+}L_{p_-}y = p_+\circ(p_- \circ y) = y \circ (p_+ \circ p_-) = 0$, then $U_{p_+,p_-} = 0$ and $J_0$ is trivial. If $J_0$ is trivial, then $\U_{p_+,p_-} = 0$ and in particular $L_{p_+}L_{p_-} = 0$. This implies that $L_{p_+}^2 = L_{p_+}$ and that $U_{p_+} = L_{p_+} = L_{p_+^2}$, and by \cite[Theorem 5]{youngson2} $p_+$ is central. Now, if $J_0$ is trivial and $p_+$ central define $v = p_+x - p_-x$, which is positive as it is the sum of two positive elements, and $\varepsilon_p = 2p_+-1$, a central symmetry as $p_+$ is a central projection. An easy computation gives us $x = v\varepsilon_p$. If $J_0$ is not trivial, we claim that $J_i$ is an invariant space by $\U_x$. To prove it, let $y\in \V$ belong to $J_+$, i.e. $y = \U_{p_+}y$. Since  $x$ and $p_+$ belong to the strongly associative Jordan subalgebra $\mathcal{C}(x)$, 
\begin{equation*}
	\U_xy = \U_x\U_{p_+}y= \U_{xp_+}y = \U_{p_+x}y = \U_{p_+}\U_xy,
\end{equation*}
thus $\U_xy$ belongs to $J_+$ and $J_+$ is invariant by $\U_x$. An analogous computation tells us that $J_-$ is also invariant. Now, let $y$ belong to $J_0$, $y = 2\U_{p_+,p_-}y$, then
\begin{align*}
\U_xy & = \U_x2\U_{p_+,p_-}y = \U_x(\U_{p_++p_-} - \U_{p_+} - \U_{p_-})y = (\U_{p_++p_-} - \U_{p_+} - \U_{p_-})\U_xy\\
&= 2\U_{p_+,p_-}\U_xy,
\end{align*}
and it follows that $\U_xy$ belongs to $J_0$ hence $J_0$ is invariant by $\U_x$.  Since we have $\U_{p_+}+\U_{p_-} + 2\U_{p_+,p_-} = Id$, then $\oplus J_i = \V$, and as each $J_i$ is invariant for $\U_x$, we have that 
\begin{equation*}
	\sigma(\U_x) = \sigma(\U_x|_{J_+}) \cup \sigma(\U_x|_{J_-}) \cup \sigma(\U_x|_{J_0}).
\end{equation*}
We have to study the restriction of $\U_x$ to each subspace.  Let us study first the restriction of $\U_x$ to $J_+$ and $J_-$: if $y$ belongs to $J_+$, then $\U_xy = \U_x\U_{p_+}y = \U_{xp_+}y$ (again, the last equality holds by \cite[Chapter 1, Section 8, Lemma 8]{jacobson}, as $x$ and $p_+$ belong to $\mathcal{C}(x)$). Define $x_+ = xp_+$, then $\U_x|_{J_+} = \U_{x_+}$. Analogously, define $x_- = xp_-$ and we have that $\U_x|_{J_-} = \U_{x_-}$. Note that the spectrum of $x_+$ is the positive component of the spectrum of $x$, and the spectrum of $x_-$ is the negative component of the spectrum of $x$. As such, $x_+$ is positive and $x_-$ is negative, and as we proved before, $\U_{x_+}$ and $\U_{x_-}$ are positive operators.  Then, $\sigma(\U_x|_{J_+}),\sigma(\U_x|_{J_-}) \subset \mathbb{R}_{>0}$. Since $x_+ + x_- = x$, and as $x$, $p_+$ and $p_-$ belong to an associative subalgebra of $\V$, we have that $x_+\circ x_- = (xp_+)\circ (xp_-) = x^2\circ (p_+\circ p_-) =0$. Now let us study the restriction of $\U_x$ to $J_0$: if $y$ belongs to $J_0$, then
\begin{equation*}
\begin{aligned}
	\U_xy &= \U_x2\U_{p_+,p_-}y = \U_x(Id - \U_{p_+} - \U_{p_-})y = (\U_x - \U_x\U_{p_+} - \U_x\U_{p_-})y \\
	&=(\U_{x_+ + x_-} - \U_{x_+} - \U_{x_-})y = 2\U_{x_+,x_-}y.
\end{aligned}
\end{equation*}
Moreover, as $x_+$, $x_-$ belong to the commutative associative subalgebra $\mathcal{C}(x)$, we have that $L_{x_+}$ and $L_{x_-}$ commute, and since $x_+\circ x_-=0$, then 
\begin{equation*}
	2\U_{x_+,x_-} = 2(L_{x_+}L_{x_-} + L_{x_-}L_{x_+} - L_{x_+\circ x_-}) = 4L_{x_+}L_{x_-}.
\end{equation*}

Suppose that $\sigma(\U_x)\subset (0,+\infty)$, in particular $\U_x$ is invertible, thus $x$ is invertible (Remark \ref{invert}). Assumming that the spectrum of $x$ has a positive and a negative value, we will arrive to a contradiction. We have to analyze the restriction of $\U_x$ to $J_0$, which is not null. Since $L_{x_+}$ commutes with $L_{p_+}$ and $L_{p_-}$, it is plain that $L_{x_+}$ commutes with $\U_{p_+,p_-}$ thus $J_0$ is invariant for $L_{x_+}$. Likewise, $J_0$ is invariant for $L_{x_-}$. Then 
\begin{equation*}
	\sigma(\U_x|_{J_0}) = 	\sigma(2\U_{x^+,x^-}|_{J_0}) =4\sigma(L_{x_+}L_{x_-}|_{J_0}) \subset 4\sigma(L_{x_+}L_{x_-}).
\end{equation*}
Now if $A,B$ are commuting elements of a Banach algebra, then $\sigma(AB)\subset \sigma(A)\sigma(B)$, and since the spectrum of $L_{x_+}$ is positive and the spectrum of $L_{x_-}$ is negative by Remark \ref{xvslx} and the fact they commute we have that 
\begin{equation*}
	\sigma(\U_x|_{J_0})  \subset 4\sigma(L_{x_+}L_{x_-}) \subset 4\sigma(L_{x_+})\sigma(L_{x_-}) \subset (-\infty,0].
\end{equation*}
As $\U_x$ is invertible and $J_0$ is not null, $\sigma(\U_x|_{J_0}) \subset (-\infty,0)$ and this shows that $\U_x$ has a negative number in its spectrum, a contradiction.
\end{proof}

\begin{rem}\label{uequis}
For $x\in \V$ invertible, if $\chi_{\pm}$ denotes the characteristic function of the positive and negative parts of $\sigma(x)$, we can write $x=x_++x_-=p_+x+p_-x$ with $p_{\pm}=\chi_{\pm}(x)\in \mathcal C(x)$. Then the previous proof shows that 
$$
\U_x=\U_{x_+}+\U_{x_-}+4L_{x_+}L_{x_-},
$$
the sum of 3 mutually disjoint operators acting on the subspaces $J_+=\ran(\U_{p_+})$, $J_-=\ran(\U_{p_-})$ and $J_0=\ran(\U_{p_+,p_-})$ respectively, with $\V=J_+\oplus J_-\oplus J_0$. Thus
$$
\sigma(\U_x) = \sigma(\U_{x_+})\sqcup \sigma(\U_{x_-})\sqcup 4\sigma(L_{x_+}L_{x_-})
$$
and moreover $\sigma(L_{x_+}L_{x_-})\subset \sigma(L_{x_+})\sigma(L_{x_-})$ since these commute.
\end{rem}

\section{The structure group and its Lie algebra}\label{s3}

We will consider a group of automorphisms that will act on $\Omega$; we want it to be a Banach-Lie group. To give an appropiate Lie and manifold structure to the group of transformations that fix the cone, we will study first a larger group and derive the structure from there. If we look at the fundamental formula (\ref{eq:fundamf}), all transformations $g\in \bv$ in the image of the quadratic representation of $\V$ have the property that $\U_{gx} = g\U_xg$ for all $x\in \V$. In other words, for each $g$ there exists another transformation (in this case the same $g$) such that the equality holds. There are other distinctive transformations that share the same property:

\begin{defi}[Structure Group]
	The structure group of $\V$ is the set of $g\in\glv$ such that there exists another $g^*\in \glv$ with
	\begin{equation*}
		\U_{gx} = g \U_x g^*
	\end{equation*}
	for all $x$ in $\V$.
	We denote the structure group as $\Str$, below we recall how is it in fact a group, and present it in a fashion that will enable the construction of its Banach manifold structure.
\end{defi}

\begin{rem}[The adjoint and group operations]
	Clearly $Id$ belongs to $\Str$, as $\U_{Idx} = \U_x = Id\U_xId$. If $g$ and $h$ belong to $\Str$, then $\U_{ghx} = g\U_{hx}g^* = gh\U_xh^*g^*$, so $gh$ belongs to $\Str$ and $(gh)^* = h^*g^*$. Finally, 
	\begin{equation*}
		\U_{g^{-1}x} = g^{-1}g\U_{g^{-1}x}g^*(g^*)^{-1} = g^{-1}\U_{gg^{-1}x}(g^*)^{-1} = g^{-1}\U_x(g^*)^{-1},
	\end{equation*}
	so $g^{-1}$ belongs to $\Str$ and $(g^{-1})^* = (g^*)^{-1}$. 

	If $g$ belongs to the structure group so does $g^*$, and $(g^*)^* = g$. To prove it, note that for $x = 1$, we have $\U_{g1} = g\U_1g^* = gg^*$, thus 
	$$
	g^* = g^{-1}\U_{g1} \qquad \forall g\in \Str.
	$$
 	$\U_{g1}$ also belongs to the structure group, hence $g^* = g^{-1}\U_{g1}$ belongs to $\Str$. Moreover,
	\begin{equation*}
	\begin{aligned}
		(g^*)^* &= (g^{-1}\U_{g1})^* = \U_{g1}^*(g^{-1})^* = \U_{g1}((g^*)^{-1}) = \U_{g1}((g^{-1}\U_{g1})^{-1}) = \U_{g1}\U_{g1}^{-1} g = g.
	\end{aligned}
	\end{equation*}
\end{rem}

\begin{rem}	Let $y$ be an invertible element in $\V$, then $\U_y$ belongs to the structure group and $\U_y^* = \U_y$. This is a direct consequence of the fundamental formula and from the fact that $y$ is invertible if and only if $\U_y$ is invertible.
\end{rem}

\begin{rem}\label{ginvertible}
If $g\in \Str$, then we know that for every $x\in \V$ we have $\U_{gx} = g\U_xg^*$. Then, for every invertible $x$, we see that $\U_{gx}$ is invertible, and so is $gx$. As we discussed before, $g^* = g^{-1}\U_{g1}$ and then it is clear that $g\in \Str$ if and only if $g(1)$ is invertible and
$$
\U_{gx}=g\U_xg^{-1}\U_{g1}\qquad \forall\, x\in \V.
$$
Note also that since for $g\in \Str$  we have $\U_{g1} = g\U_1g^* = gg^*$, then by Theorem \ref{upositive} $\sigma(gg^*)\subset (0,+\infty)$ if and only if $g(1) = v\varepsilon_p$ for $v \in \Omega$ and $\varepsilon_p$ central symmetry.
\end{rem}

\begin{lema}\label{esalge}Let $g\in \glv$. Then $g\in \Str$ if and only if 
$$
\U_{gx}=g\U_xg^{-1}\U_{g1}  \quad \textrm{ and }\quad \U_{g^{-1}x}=g^{-1}\U_xg\U_{g^{-1}(1)} \qquad \forall \, x\in \V.
$$
\end{lema}
\begin{proof}
As we discussed above, when $g\in \Str$, then $g(1)$ is invertible and we have $g^*=g^{-1}\U_{g1}$; moreover $g^{-1}\in \Str$ thus $g$ obeys both equations of the lemma. Now assumme that $g\in \glv$ obeys both equations. In particular replacing $x=g^{-1}(1)$ in the first equation and $x=g(1)$ in the second we have
$$
1=\U_1=g\U_{g^{-1}(1)} g^{-1}\U_{g1}\quad  \textrm{ and } \quad 1=\U_1=g^{-1}\U_{g1} g\U_{g^{-1}(1)}.
$$
This tells us $g^{-1}\U_{g1}$ is invertible (with inverse $g\U_{g^{-1}(1)}$), thus $\U_{g1}$ is invertible, hence $g(1)$ is invertible. Now if we define $g^*=g^{-1}\U_{g1}$ it is clear that $g^*\in \glv$, and on the other hand because of the first equation $\U_{gx}=g\U_x g^*$, thus $g\in \Str$.
\end{proof}

\begin{rem}[Complexification of operators in the Structure Group]\label{gcestr}Let $g\in \Str$, let $g^{\mathbb C}(a+ib)=ga+igb$ be its complexification to $\V^{\mathbb C}$, then $g\in \glv$  and, denoting $g^{\mathbb C}=g$ for short, we have
\begin{align*}
\U_{g(a+ib)}& =\U_{ga+igb}=\U_{ga}+\U_{igb}+2 \U_{ga,igb}=\U_{ga}-\U_{gb}+2i\U_{ga,gb}\\
&= \U_{ga}-\U_{gb}+i(\U_{ga+gb}-\U_{ga}-\U_{gb})\\
&= g\U_ag^*-g\U_bg^*+i(g\U_{a+b}g^*-g\U_ag^*-g\U_bg^*)\\
& = g\U_ag^*-g\U_bg^*+i(g\U_ag^*+g\U_bg^*+2g\U_{a,b}g^*-g\U_ag^*-g\U_bg^*)\\
&=g\U_ag^*-g\U_bg^*+2ig\U_{a,b}g^*=g(\U_a-\U_b+2i\U_{a,b})g^*\\
&=g\U_{a+ib}g^*.
\end{align*}
Thus $g^{\mathbb C}$ belongs to $\texttt{Str}(\V^{\mathbb C})$
\end{rem}

\begin{defi}An element $k\in \glv$ is a \textit{multiplicative automorphism} of $\V$ (or an automorphism of $\V$ for short) if $k(a\circ b)=k(a)\circ k(b)$ for all $a,b\in \V$. We will denote this set with $\Aut$. 
\end{defi}

Note that $k(L_xy)=k(x\circ y)=(kx)\circ (ky)=L_{kx}ky$, then
$$
kL_vk^{-1}=L_{kv} \quad\textrm{ and }\quad \U_{kv}=k\U_vk^{-1}\quad \forall\, v\in \V.
$$
In particular $\Aut$ preserves the invertibles and $k(v)^{-1}=k(v^{-1})$. Moreover $\Aut$ preserves the cone $\Omega$, since $k(v^2)=(kv)^2$ and the cone can be characterized as the set of invertible squares in $\V$ (Remark \ref{expsq}).

\begin{rem}[$\Aut$ is a subgroup of $\Str$ and if $k\in \Aut$, then $k^*=k^{-1}$]\label{rem:autinv} This can be seen as follows: $k$ is a multiplicative automorphism, then 
		\begin{equation*}
			k\U_x(b) = k(2x\circ(x\circ b) - x^2\circ b) = 2k(x)\circ(k(x)\circ k(b)) - k(x)^2\circ k(b) =\U_{k(x)}(k(b))
					\end{equation*}
thus $k\U_x = \U_{kx}k$, and if we set $k^*=k^{-1}$ we clearly have $\U_{kx} = k\U_xk^{-1}=k\U_xk^*$ which proves the assertions.
\end{rem}

The following well-known characterization will be used repeatedly. Since we are interested in both the real and complex cases, we include a proof: 

\begin{lema}\label{auto1} Let $k\in\Str$, then $k\in \Aut$ if and only if $k1=1$. Likewise, if $k\in\texttt{Str}(V^{\mathbb C})$, then $k\in \texttt{Aut}(V^{\mathbb C})$ if and only if $k1=1$. 
 \end{lema}
\begin{proof}
	Suppose $k\in \Aut$, then $k\in \Str$ by the previous remark. From $\U_{k1}=k\U_1k^{-1}=Id$ we see that $\U_{k1}$ is invertible thus  $k(1)$ is invertible (Remark \ref{invert}). Now $k(1)=k(1^2)=k(1)^2$,  thust $k(1)=1$.

	Now, suppose $k \in \Str$, $k(1) = 1$. We have that $k^* = k^{-1}\U_{k1} = k^{-1}$. Then, $\U_{kx} = k\U_xk^{-1}$. Note that as $k(1) = 1$, $k^{-1}(1) = 1$. Then
	\begin{equation*}
		k(x)^2 = \U_{kx}(1) = k\U_xk^{-1}(1) = k(x^2).
	\end{equation*}
	Linearizing, we have that $k(x\circ y) = k(x) \circ k(y)$, and $k$ belongs to the automorphism group. The proof for the complex case is identical.
\end{proof}

The following theorem will be proven later. The ``only if'' part is nontrivial in infinite dimension and it follows from the caracterization of positive $\U_x$ obtained in Theorem \ref{upositive}:

\begin{teo}\label{caractautomorf}
	Let $g\in \Str$. Then $g^* = g^{-1}$ if and only if $g =L_{\varepsilon} k$, where $k\in \Aut$ and $\varepsilon$ is a central symmetry of $\V$.
\end{teo}

\begin{defi}[Inner Structure group]
	The Inner Structure Group of $\V$ is the set
	\begin{equation*}
		\texttt{Inn\!}\Str = \langle \U_x \rangle_{x\in \V \textrm{ invertible}}.
	\end{equation*}
$\texttt{Inn\!}\Str$ is obviously a subgroup of $\Str$. If $g\in \Str$ and $x\in \V$ invertible, we have that $	\U_{gx} = g\U_xg^* = g\U_xg^{-1}\U_{g1}$. 	Then $g\U_xg^{-1} = \U_{gx}\U_{g1}^{-1}$, which belongs to $\texttt{Inn\!}\Str$, thus $\texttt{Inn\!}\Str$ is a normal subgroup.
\end{defi}

\subsection{The Banach-Lie group $\Str$ and its Banach-Lie algebra}

We will now establish a differentiable structure for $\Str$. As $\V$ is a Banach space, we can give the operator space the supremum norm, wich makes it also a Banach space. This way, the invertible operators form a Lie group with Lie algebra $\mathfrak{gl}(\V) = \bv$. Moreover, in this norm the quadratic representation is continuous, so the condition for the structure group is a closed one. Then, the structure group is a closed subgroup of $\glv$. We will see that $\Str$ is a Lie subgroup (i.e. and embedded closed submanifold of $\glv$ with its Lie group structure). This is non-trivial if the dimension of $\V$ is infinite, since Lie's closed subgroup theorem does not hold for infinite dimensional Banach-Lie groups.

\begin{rem}\label{expo} The following facts will be useful soon: 
\begin{enumerate}
\item if $g_t$ is a smooth path in $\glv$, then from $(g_t)^{-1}g_t=Id$, differentiating with respect to $t$ it is plain that $\frac{d}{dt}(g_t)^{-1}=-(g_t)^{-1}g_t'(g_t)^{-1}$. In particular if $g_0=Id$ and $g_0'=H\in \bv$, then $(g^{-1})_{t=0}'=-H$.
\item If $G$ is a Banach-Lie group, then the exponential map of $G$ is a local diffeomorphism around $0\in \lie(G)$, the Banach-Lie algebra of $G$, and moreover $D_0(\exp)=Id_{\lie(G)}$.
\item If $H\subset G$ is a Banach-Lie subgroup of $G$, we define $\lie(H)=T_1H\subset T_1G=\lie(G)$. Then $\lie(H)$ is a Banach-Lie subalgebra of $\lie(G)$ and 
$$
\lie(H)=\{v\in \lie(G): \exp(tv)\subset H \quad \forall\, t\in\mathbb R\}.
$$
The inclusion $\supset$ is clear. On the other hand if $v=g_0'$ with $g_t\subset H$ and $g_0=1$, then since the exponential map of $G$ is a local diffeomorphism around $1\in G$, we can lift it to a smooth path  $\Gamma\subset \lie(G)$ with $\Gamma_0=0$, hence $\gamma_t=e^{\Gamma_t}$ (note that $v=\gamma_0'=D(\exp)_0\Gamma_0'=\Gamma_0'$). Then 
$$
e^{tv}=e^{t\Gamma_0'}=\lim_n e^{n\Gamma(t/n)}=\lim_n \gamma(t/n)^n \in H
$$
since $\gamma(t/n)\in H$ for each $t,n$, and $H$ is a closed subgroup.
\end{enumerate}
\end{rem}

\begin{defi}
A subgroup  $H\subset \glv$ is an \textit{algebraic subgroup} if there exists a family of polynomials in two variables $P_i=P_i(A,B)$ with $P_i:\B(\V)\times \bv\to \B(\V)$ such that $P_i(h,h^{-1})=0$ for each of the polynomials $P_i$ and every $h\in H$. The algebraic subgroup Theorem  \cite[Theorem 4.13]{beltita} states that any algebraic subgroup of $\glv$ is in fact a Banach-Lie subgroup (closed, embedded and with complemented Lie algebra).
\end{defi}

\begin{teo}
	$\Str$ is an algebraic subgroup of $\glv$. Then $\Str$ is an embedded Lie group, and if $\U = \{X \in \bv: \|X\| < \frac{\pi}{3}\}$, then 
	$$
	\exp(\U \cap \lie(\Str)) = \exp(\U) \cap \Str.
	$$ 
	Moreover  $\lie(\Str) = \str$, with
	\begin{equation*}
		\str = \{H \in \bv \text{ such that } \exists\, \overline{H}\in \bv: 2\U_{x,Hx} = H\U_x - \U_x\overline{H}\quad \forall\, x\in \V \}.
	\end{equation*}
\end{teo}
\begin{proof}
 Define for each $x\in \V$ and $A,B\in \bv$ the polynomials
	\begin{equation*}
		P_x(A,B) = \U_{Ax} - A \U_x B \U_{A1} \qquad		Q_x(A,B) = \U_{Bx} - B \U_x A \U_{B1}.
	\end{equation*}
By Lemma \ref{esalge} we know that $g\in \glv$ belongs to the structure group if and only if $P_x(g,g^{-1}) = 0=Q_x(g,g^{-1})$ for all $x\in \V$, so  $\Str$ is an algebraic sugroup of $\glv$. Then, it is an embedded Lie group and the assertion on the neighbourhoods follows from the fact that the degree of the polynomials is 3  (see \cite[Theorem 4.13]{beltita}). Let $g_t$ be a smooth path in the structure group $\Str$, with $g_0 = Id$, $g_0' = H\in \lie(\Str)$. Let's compute the differential of $g^*$. As $g^* = g^{-1}\U_{g1}$,
	\begin{equation*}
	\begin{aligned}
		({g^*})'_0 &= (g^{-1})_0' \U_{g_0 1} + g^{-1}_0 (\U_{g1})_0' = -H + D_{g_0 1}(\U_{-})g_0'1 \\
		&=-H + 2 \U_{-,g_0 1}H1 = -H + 2\U_{H1,1}
	\end{aligned}
	\end{equation*}
 by equation (\ref{du}). 	As $g_t$ is in the structure group for every $t$, we have that $\U_{g_t x} = g_t\U_xg_t^*$. Differentiating this equality, we obtain
	\begin{equation*}
		\frac{d}{dt}(\U_{g_tx})|_{t=0} = D_{g_0 x}(\U_-)g_0'x = 2\U_{x,Hx}
	\end{equation*}
	on the left side and $g_0' \U_xg^*_0 + g_0 \U_x (g^*)'_0 = H\U_x + \U_x(-H+ 2\U_{H1,1})$	on the right side. Then if we take $\overline{H} = H -2\U_{H1,1}$ it follows that $2\U_{x,Hx} = H\U_x - \U_x\overline{H}$, which shows the inclusion $\subset$. Now, assume that  $H\in \bv$ is in the right-hand set, by the previous remark it suffices to show that $e^{tH}$ belongs to the structure group for all $t\in\R$. Consider $f_t = \U_{e^{tH}x}$ and $g_t = e^{tH}\U_x(e^{tH})^* = e^{tH}\U_x e^{-tH}\U_{e^{tH}1}$. We want to see that both functions are equal.  First, notice that $f_0 = \U_x = g_0$. Moreover, we have
	\begin{equation*}
	\begin{aligned}
		f_t' &= D_{e^{tH}x}(\U_-)e^{tH}Hx = 2 \U_{e^{tH},e^{tH}Hx} = 2 \U_{e^{tH},He^{tH}x} = H\U_{e^{tH}x} - \U_{e^{tH}x}\overline{H} \\
		&=Hf_t - f_t\overline{H},
	\end{aligned}
	\end{equation*}
	and
	\begin{equation*}
	\begin{aligned}
		g_t' &= e^{tH}H\U_xe^{-tH}\U_{e^{tH}1} - e^{tH}\U_xe^{-tH}H\U_{e^{tH}1} + e^{tH}\U_xe^{-tH}D_{e^{tH}1}(\U_-)e^{tH}H1 \\
		&= e^{tH}H\U_xe^{-tH}\U_{e^{tH}1} - e^{tH}\U_xe^{-tH}H\U_{e^{tH}1} + 2e^{tH}\U_xe^{-tH}\U_{e^{tH}1,He^{tH}1} \\
		&= e^{tH}H\U_xe^{-tH}\U_{e^{tH}1} - e^{tH}\U_xe^{-tH}H\U_{e^{tH}1} + e^{tH}\U_xe^{-tH}(H\U_{e^{tH}1} - \U_{e^{tH}1}\overline{H}) \\
		&=He^{tH}\U_xe^{-tH}\U_{e^{tH}1} - e^{tH}\U_xe^{-tH}\U_{e^{tH}1} \overline{H} \\
		&=Hg_t - g_t \overline{H}.
	\end{aligned}
	\end{equation*}
	Then, by the uniqueness of solutions of ordinary differential equations in the Banach space $\bv$, $f_t= g_t$ for all $t\in \mathbb R$.
\end{proof}

From the previous proof we know that if $H\in \str$ then $\overline{H}=H-2\U_{H1,1}$. We will now recall how $\str$ is the direct sum of two distinct subspaces.

\begin{rem}[$L$ operators] By Lemma \ref{expu}, the exponential of the left-multiplication operators are $\U$-operators. Moreover, as for every $t$ and every $v\in \V$ we have that $e^{tL_v}\in  \Str$, the left-multiplication operators $L_v$ belong to $\str$. The latter is a Banach subspace of $\bv$, which we denote  $\mathbb{L} = \{L_v, v \in \V\}\subset \str$.
\end{rem}

\begin{rem}[Derivations] Let $\der$ be the subspace of \textit{derivations} i.e. $D\in \bv$ such that $D(x\circ y)=Dx\circ y+x\circ Dy$ for all $x,y\in \V$. It is plain that $\der\subset \bv$ is a Banach-Lie subalgebra.
\end{rem}

\begin{teo}\label{autisbanachlie}
	$\Aut\subset \Str$ is an algebraic subgroup of $\glv$, in particular it is an embedded Banach-Lie subgroup, and if $\U = \{X \in \bv: \|X\| < \frac{\pi}{2}\}$, then 
	$$
	\exp(\U \cap \lie(\Aut)) = \exp(\U) \cap \Aut.
	$$ 
Moreover for its Banach-Lie algebra $\aut=\lie(\Aut)\subset \str$ we have $\aut=\der$.
\end{teo}
\begin{proof}
Define for each $x,y\in \V$ the polynomial $P_{x,y}(A) = A(x \circ y) - A(x) \circ A(y)$. The group of automorphisms is the intersecction of the zeros of every $P_{x,y}$. Then, it is an algebraic group and therefore an embedded Lie group: this fact and the assertion of the neighbourhoods follows from the algebraic subgroup theorem \cite[Theorem 4.13]{beltita}), noting that the polynomials have degree 2. Since $\Aut\subset \Str$ it is plain that $\aut=\lie(\Aut)$ is a Banach-Lie subalgebra of $\str$. Now let $\gamma$ be a path of automorphisms, $\gamma_0 = Id$, $\gamma_0'= D$. As $\gamma_t$ is an automorphism for every $t$, we have that $\gamma_t (x\circ y) = \gamma_t x \circ \gamma_t y$. Differentiating this, we obtain
	\begin{equation*}
		\gamma_t'(x\circ y) = \gamma_t'x \circ \gamma_t y + \gamma_t x \circ \gamma_t'y,
	\end{equation*}
	and for $t=0$ we have $D(x\circ y) = Dx \circ y + x \circ Dy$. 	So, $D$ is a derivation and we have $\aut\subset \der$. Now take $D$ a derivation and define $f_t = (e^{tD}x)\circ(e^{tD}y)$. We have that $f_0 = x \circ y$. Moreover,
	\begin{equation*}
		f_t' = e^{tD}Dx \circ e^{tD}y + e^{tD}x \circ e^{tD}Dy = (De^{tD}x) \circ e^{tD}y + e^{tD}x \circ (De^{tD}y) = Df_t.
	\end{equation*}
	Then $f_t = e^{tD}(x\circ y)$ by the uniqueness of solutions of ordinary differential equations, and $e^{tD}$ is an automorphism for every $t$. Then by Remark \ref{expo}, $D$ belongs to the Lie algebra $\aut$. This proves that $\aut = \der$.
\end{proof}

\begin{rem}If we let $\Aut_0$ be the component of the identity, we have that $\Aut_0\subset \Aut$ is open and it is generated by exponentials of derivations,  $\Aut_0=\langle e^D\rangle_{D\in \der}$. With the previous theorem it is possible to give a different characterization of the derivations in $\str$:  as the group of automorphisms is a Lie subgroup of the structure group, $\der$ is a Lie subalgebra of $\str$. Moreover, $D(1) = D(1\circ1)= 2D(1)$, so $D(1) = 0$. Now, if $D$ belongs to $\str$ and $D1 = 0$,
	\begin{equation*}
		e^{tD}1 = 1 + tD1 + \frac{t^2}{2}D^21 + \cdots = 1.
	\end{equation*}
	Then, for every $t$, $e^{tD}$ is a transformation in the structure group that sends $1$ to $1$, so it is an automorphism. This, by Remark \ref{expo}, implies that $D$ is a derivation, thus  
	$$
	\der=\{D\in\str : D1=0\}.
	$$
	We have seen that $\der$ and $\mathbb{L}$ are contained in $\str$, then their sum is as well. Take $X$ in $\str$ and $u = X1$. Consider $Y = X - L_u$, then $Y$ belongs to $\str$ and $Y1 = X1 - u = 0$, so $Y$ is a derivation. 	Now suppose $X$ belongs to $\mathbb{L}$ and is a derivation. Then, $X1 = L_u1 = u$ for some $u$, but as $X$ is a derivation, $X1= 0$. Then $u= 0$ and this shows that $\str=\der \oplus\, \mathbb{L}$.
\end{rem}

%\begin{rem}[The Lie Triple System] The condition for a derivation can be written as: $D$ is a derivation if $[D,L_x] = L_{Dx}$ for all $x$. If $D_1$ and $D_2$ are derivations, then $[D_1,D_2]$ is a derivation as well. Finally, by \cite[Proposition II.4.1.(i)]{faraut},
%\begin{equation*}	[[L_x,L_z],L_y] = L_{x(yz) - (xy)z}.\end{equation*}
%This also tells us that $[L_x,L_z]$ is a derivation: if we call $D = [L_x,L_z]$, the last equality can be written as $[D,L_y] = L_{Dy}$. To summarize, we have a decomposition of $\str$ as $\mathbb{L} \oplus \der$ such that 
%\begin{equation}\label{eq:LieTripleSystem}
%	[\der,\mathbb{L} ] \subset \mathbb{L} ; \hspace{10pt} [\der,\der] \subset \der;  \hspace{10pt} [[\mathbb{L} ,\mathbb{L} ],\mathbb{L} ] \subset \mathbb{L}.\end{equation}
%This is known as a \textit{Lie triple system}, and $\der\subset str(v)$ is a \textit{Cartan subalgebra}.
%\end{rem}

\subsection{The group $\GO$ preserving the cone $\Omega$}

One can make $\Str$ act on the cone of squares $\Omega$. But if $x\in\Omega$ and $g$ belongs to $\Str$, although $g(x)$ is invertible, it is not necessarily true that $g(x)$ belongs to $\Omega$, in fact

\begin{lema}\label{nointerseca}
Let $g\in \Str$. Then $g(\Omega)$ and $\Omega$ are equal or do not intersect. 
\end{lema}
\begin{proof}
Define the set $A_g = \{x \in \Omega: g(x) \in \Omega\}$. As $A_g = g^{-1}(\Omega) \cap \Omega$, we have that $A_g$ is an open set in $\Omega$. Moreover, it is also closed in $\Omega$: let $\{x_n\}\subset A_g$ such that $x_n$ converges to $x$ in $\Omega$. Then, $g(x_n)$ converges to $g(x)$, and as $g(x_n)$ belongs to $\Omega$ for all $n$, we have that $g(x)$ belongs to $\overline{\Omega}$. But as $g$ belongs to the structure group, $g(x)$ is invertible, so $g(x)$ belongs to $\Omega$ thus $x\in A_g$. As $A_g$ is open and closed in the convex set $\Omega$, it is either empty or $\Omega$, which proves the lemma.
\end{proof}

%We now take a look to the spectrum of $\U_{gx}$ for $x\in \Omega$ and $g\in\Str$. By Theorem \ref{upositive} and the fact that $\U_y$ preserves the cone (see Remark \ref{autinng} below), the following is a spectral version of the previous lemma:
%\begin{lema}\label{g1pos}
%Let $g\in \Str$, let $z=g^{-1}(1)$. Then $\U_z^{-1}=g^{-1}\U_{g1}g$ hence $\sigma(\U_z^{-1})=\sigma(\U_{g1})$. Moreover for any $x=y^2\in\Omega$, we have $\sigma(\U_{gx})=\sigma(\U_{\U_y(z^{-1})})$. 
%\end{lema}
%\begin{proof} Since $g\in \Str$ then $1=\U_1=\U_{g^{-1}(g1)}=g^{-1}\U_{g1}g\U_{g^{-1}(1)}$, thus $g^{-1}\U_{g1}g=\U_{g^{-1}(1)}^{-1}$. Let $z=g^{-1}(1)$; then since $\sigma(AB)=\sigma(BA)$ for invertible operators, 
%$$\sigma(\U_z^{-1})=\sigma(\U_{g^{-1}(1)}^{-1})=\sigma(g^{-1}\U_{g1}g)=\sigma(\U_{g1}).$$
%Now let $x\in \Omega$, then $x=y^2$ for some $y\in \Omega$ and $\U_x=(\U_y)^2$. By the property of the spectrum of a product and the fundamental identity, we have
%$$\sigma(\U_{gx})=\sigma(g\U_y^2g^{-1}\U_{g1})=\sigma(\U_y g^{-1}\U_{g1}g\U_y)=\sigma(\U_y \U_{g^{-1}(1)}^{-1}\U_y)=\sigma(\U_{\U_y(z^{-1})}).\qedhere$$
%\end{proof}

We now recall that $\GO$ is defined as the subgroup of all invertible transformations of $\V$ preserving the cone; we will see later that it is a subgroup of $\Str$.

\begin{defi}
Let $\GO$ be the set of isomorphisms of $\V$ that preserve the cone:
$$
\GO=\{g\in \glv: g(\Omega)=\Omega\}.
$$
\end{defi}

\begin{rem}[$\texttt{Inn\!}\Str$ and $\Aut$ are subgroups of $\GO$]\label{autinng}
If $x\in \V$ and $y\in\Omega$,  then $\U_x(y)\in {\overline{\Omega}}$, a proof can be found in \cite[Proposition 3.3.6]{stormer}. But as $x$ and $y$ invertible implies $\U_xy$ invertible (Remark \ref{invert}) we have that $\U_xy \in \Omega$. Thus $\texttt{Inn\!}\Str\subset\GO$. Moreover, we have that if $g\in \Aut$, then $g(x^2) = g(x)^2$, so $g$ belongs to $\GO$ because $\Omega=\{v^2: v\in \V\,\textrm{ invertible}\}$ (Remark \ref{expsq}), thus $\Aut\subset \GO$ also.
\end{rem}

It follows from the characterization of isometries that every automorphism is an isometry, we recall these well-known facts here:

\begin{prop}Let $\Omega\subset\V$ be its positive cone. Then
\begin{enumerate}
\item Every linear transformation $T$ that maps $\Omega$ to itself is continuous and $\|T\| = \|T(1)\|$. 
\item If $g$ belongs to $\GO$, then $g$ is an isometry if and only if $g(1) = 1$.
\item 	Every automorphism is an isometry, and every isometry in $\GO$ is an automorphism.
\end{enumerate}
\end{prop}
\begin{proof}
The first two assertions are proved in \cite{chu} (Lemma 2.2 and Proposition 2.3 respectively). For every $g\in\Aut$ we have that $g\in \GO$ and $g(1) = 1$, so $g$ is an isometry. On the other hand every surjective linear isometry between two JB-algebras that maps the identity to the identity is an automorphism by \cite[Theorem 4]{wright}, so every isometry in $\GO$ is an automorphism.
\end{proof}

Combining these results, there is a useful characterization of $\GO$ following from \cite[Theorem III.5.1]{faraut}, we include a proof for completeness:

\begin{prop}\label{gigualuk}
	Every $g$ in $\GO$ can be written as $g = \U_yk$, where $y$ belongs to $\Omega$ and $k$ is an automorphism, and $\GO\subset \Str$.
\end{prop}
\begin{proof}
	As $g$ belongs to $\GO$ and $1$ belongs to $\Omega$, $g(1)$ is positive, so $g(1) = y^2$, with $y$ positive. Let $k = \U_y^{-1}g$. We want to see that $k$ is an isometry. But
	\begin{equation*}
		k(1) = \U_y^{-1}g(1) = \U_y^{-1}(y^2) = 1.
	\end{equation*}
	As $\U_y$ and $g$ belong to $\GO$, so does $k$, and as $k(1) = 1$, $k$ is an isometry. From the last proposition, $k$ is an automorphism.  For the assertion on the inclusion, note that since $\texttt{Inn\!}\Str$ and $\Aut$ are subgroups of $\Str$, the previous results tells as that $\GO$ is a subgroup of $\Str$. 
\end{proof}

\begin{rem}[$\GO$ is a closed subgroup of $\Str$]
	If $x\in\overline{\Omega}$ then there exists $(x_n)_{n\in \mathbb{N}} \in \Omega$ such that $x_n$ converges to $x$. Then $g(x_n)$ converges to $g(x)$, who then belongs to $\overline{\Omega}$. Conversely, as for every $g \in \Str$ and invertible $x$ we have that $g(x)$ is invertible, then if $g(\overline{\Omega}) = \overline{\Omega}$ then $g \in \GO$. In summary, if $g \in \Str$, $g\in \GO$ if and only if $g(\overline{\Omega})=\overline{\Omega}$. Thus $\GO$ is a closed subgroup of $\Str$. 
\end{rem}

For $x\in \Omega$ we let $\ln(x)\in \V$ be the unique logarithm of $x$, this is a diffeomorphism $\ln:\Omega\to \V$ by Remark \ref{expsq}.

\begin{teo}\label{retract} Let $F:\GO\times [0,1]\to \GO$ be $F(g,t)=\U_{e^{-t\ln(g1)/2}}\cdot g$. Then $\Aut$ is a strong deformation retract of $\GO$ by means of $F$. In particular  $\GO$ and $\Aut$ have the same number of connected components.
\end{teo}
\begin{proof}
It is plain that $F$ is continuous since evaluation, the logarithm, and the map $x\to \U_x$ are continuous. Clearly $F(g,0)=g$, and 
$$
F(g,1)(1)=(\U_{\sqrt{g(1)}})^{-1} g(1)=1
$$ thus $F(g,1)\in\Aut$ for each $g\in \GO$. Finally if $k\in \Aut$ then $F(k,t)=k$ since $k(1)=1$.
\end{proof}

\subsubsection{The inclusion $\GO\subset\glv$ is of embedded Banach-Lie groups}

Although $\GO$ is a closed subgroup of the structure group, it does not automatically inherit a differentiable structure for infinite dimensional $\V$. To prove that $\GO$ is a submanifold, we will see that it is an open subgroup of $\Str$. To this end, we will study $\Str_0$, the connected component of the identity of $\Str$ (which is open since $\Str$ is a Lie group) and we will see that $\Str_0$  is contained in $\GO$. Thus $\Str_0\subset \GO\subset \Str$ and each inclusion is open (moreover each inclusion is closed since open subgroups of topological groups are closed).

\begin{prop}
	Every element $g\in \Str_0$ can be written as $g = \U k$, where $\U$ belongs to the inner structure group and $k\in \Aut_0$.
\end{prop}
\begin{proof}
	Define $\varphi: \str=\mathbb L\oplus \der \to \Str$ by  $\varphi(L_x + D) = e^{L_x}e^D$. It is obviously a smooth map. Now, let $L(t)+ D(t)$ be such that $L(0) = D(0) = 0$, $L'(0) = L_x$, $D'(0) = D$. Then,
	\begin{equation*}
		D_0\varphi(L + D) = (\varphi(L(t) + D(t)))'(0) = e^{L(0)} L e^{D(0)} + e^{L(0)}e^{D(0)}D = L + D.  
	\end{equation*}
	Then, $D_0\varphi = Id$, and $\varphi$ is a local difeomorphism around $(0,0)$. 	This tells us that in a neighbourhood of the identity in $\Str$, every element can be written as $e^{L_x}e^D$. We have seen before that the exponential of a left multiplication gives us a quadratic operator, and the exponencial of a derivation gives us an automorphism. Then, in a neighbourhood of the identity every element $g$ can be written as $g = \U_xk$, with $k$ an automorphism.  In a topological group,  a neighbourhood of the identity generates the connected component of the identity. Then, every $g$ in this connected component can be written as a multiplication of elements $\U_xk$. As $k$ is an automorphism, we have that $	\U_xk_1\U_yk_2 = \U_x\U_{k_1y}k_1k_2$, which proves the claim.
\end{proof}

\begin{coro}
	The identity component $\Str_0$ is contained in $\GO$.
\end{coro}
\begin{proof}
	As the inner structure group and the automorphisms group are subgroups of $\GO$, if $g\in \Str_0$ then by the previous theorem $g=\U k\in \texttt{Inn\!}\Str\cdot \Aut\subset \GO$.
\end{proof}

Then by Proposition \ref{gigualuk}, and with a similar proof than Theorem \ref{retract},  we have that in fact

\begin{coro}\label{ref:decompUAut}
	Every element $g\in \Str_0$ can be written as $g = \U_xk$, where $x\in \Omega$ and $k=e^{D_1}e^{D_1}\dots e^{D_n}\in \Aut_0$ with $D_i\in \der$.
\end{coro}

\begin{teo}
$\GO$ is an embedded Lie subgroup of $\glv$ with $\lie(\GO) = \str$. We have $\GO=\bigsqcup_i  \Str_0\cdot k_i$, where each $k_i$ belongs to a different connected component of $\Aut$.
\end{teo}
\begin{proof}
	As the connected component of the identity $\Str_0$ is contained in $\GO$ and it is open, we have that $\GO$ is the union of translations of this component, so it is also open. Then, $\GO$ is an embedded Lie subgroup of $\Str$ and therefore of $\glv$, and since $\GO$ is open in $\Str$ we have $\lie(\GO) = \str$. Finally, we have that $\GO=\bigsqcup_i  \Str_0\cdot g_i$ with disjoint copies and $g_i\in \GO$ in different connected components of $\GO$; by Proposition \ref{gigualuk}, $g_i=\U_{x_i}k_i$ so we can assimmilate $\U_{x_i}$ to the set $\Str_0$ and this finishes the proof.
\end{proof}

\begin{rem}
	It seemed  unknown (see \cite[pag. 363]{chu}) that $\GO$ as a Lie group has the norm topology of $\bv$. 
\end{rem}

\subsection{Jordan homotopes and the components of $\Str$}

We can write $\Str=\bigsqcup_j g_j \GO$ as a disjoint union of copies of $\GO$ (here each $g_j\in \Str$ does not belong to $\GO$). We want to know how many copies of $\GO$ there are in $\Str$; we know we have at least two, as $-Id$ is an element of the structure group but does not belong to $\GO$.

Let $g$ be an element of the structure group, then it is easy to see that $g(\Omega)$ is a convex cone, as $\Omega$ is one.

\begin{lema}\label{disjuntos}
	Let $g$, $h$ belong to $\Str$. Then the convex cones $g(\Omega)$ and $h(\Omega)$ are equal or do not intersect.
\end{lema}
\begin{proof}
It is enough to see that if $g$ belongs to $\Str$, then $g(\Omega)$ is either $\Omega$ or does not intersect $\Omega$, as if $g$ and $h$ belong to the structure group, so does $h^{-1} \circ g$, and comparing $h^{-1} \circ g$ with the identity gives us the general result. But this was proved in Lemma \ref{nointerseca}.
\end{proof}

\begin{rem}[For $g,h\in \Str$, we have $g(\Omega) = h(\Omega)$ if and only if the coclass $g\GO$ equals the coclass $h\GO$] This can be seen as follows: if the coclasses are equal, we have that $h^{-1}g$ belongs to $\GO$ and $h^{-1}g(\Omega) = \Omega$. Then, $g(\Omega) = h(\Omega)$. 	Conversely, if the cones $g(\Omega)$ and $h(\Omega)$ are equal, then $h^{-1}g(\Omega) = \Omega$ and $h^{-1}g$ belongs to $\GO$. Then, the coclasses are equal.
\end{rem}

Then, we have a family of cones $\{g(\Omega): g \in \Str\}$ which are either equal or do not intersect, and by the last observation we have as many copies of $\GO$ in the structure group as different cones in that family. The two obvious cones are $\Omega$ and $-\Omega$. We now give a characterization of these copies by means of central projections of $\V$. 

\begin{lema}[Central projections]\label{centrals} Let $p\in\V$ be a central projection, then $L_p=(L_p)^2=\U_p$, $L_{px}=L_pL_x=L_xL_p$ and $px^2=(px)^2=(px)x$ for any $x\in\V$.
\end{lema}
\begin{proof}
By \cite[Theorem 5]{youngson2}, we have $L_p=L_{p^2}=\U_p$, and from it follows that  $L_p^2=\U_p^2=\U_{p^2}=\U_p=L_p$. Now $\U_{px}=\U_{L_px}=\U_{\U_px}=\U_p\U_x\U_p=\U_p^2\U_x=\U_p\U_x=L_p\U_x$ by the fundamental formula and the fact that $L_p$ commutes with $L_x,L_{x^2}$. If we apply this identity to $v=1$, we get $(px)\circ (px)=p\circ x^2$. Polarizing this identity, it follows that $(px)(py)=p(xy)$ for any $x,y\in \V$. Note that this tells us $L_p$ is a Jordan morphism. Hence $L_{px}L_py=L_pL_xy$.

As $q=1-p$ is also a central projection, last equality also holds for $q$, $L_{qx}L_qy=L_qL_xy$. Adding these two identities, we get
$$
L_x=L_{px}L_p+L_x-L_xL_p-L_{px}+L_{px}L_p,
$$
which tells us that $2L_{px}L_p=L_pL_x+L_{px}$. Hence  $2L_xL_p=2L_{px}L_p=L_pL_x+L_{px}$, and cancelling we conclude that $L_xL_p=L_{px}$. Then, $p(xy) = (px)y = x(py) = (px)(py)$.

Now $(px)^2=px^2=L_pL_xx=L_{px}x=(px)x$ and this finishes the proof.
\end{proof}

\begin{rem}
The assertions of the previous lemma are essentially in \cite[Section 2.5]{stormer}. Since $x(py)=p(xy)$ for each $x,y\in \V$, we have that $\texttt{I}_{p}=p\V=\U_p(\V)$ is a Jordan Ideal of $\V$ for each central projection $p\in \V$. It is not hard to see that for any idempotent $p\in \V$, the space $p\V$ is an ideal if and only if $p$ is central \cite[2.5.7]{stormer}. For $JBW$-algebras (JB-algebras with predual space), central projections are in one-to-one correspondence with Jordan ideals of $\V$, which are of the form $p\V$ for some central projection $p\in\V$ (see \cite[Proposition 4.3.6]{stormer}).
\end{rem}

Let $p\in\V$ be a central projection, let $\varepsilon_p$ be the central symmetry $2p-1$. By the previous lemma  $L_p$ is an idempotent of $\bv$. Let $S_p=L_{\varepsilon_p} = 2L_p-1$, then $S_p$ is a symmetry of $\bv$,  i.e. $S_p^2=1$.

\begin{lema}\label{LepsIsStr} Let $p\in \V$ be a central projection, then $S_p=2L_p-1=L_{\varepsilon_p}\in \Str$.
\end{lema}
\begin{proof}
We first compute $(S_pz)^2=(2pz-z)(2pz-z)=4(pz)^2-4(pz)z+z^2=z^2$ by the previous lemma. On the other hand
$$
L_zL_{\varepsilon_p}=L_{\varepsilon_p}L_z=2L_pL_z-L_z=L_{2pz}-L_z=L_{2pz-z}=L_{\varepsilon_pz}=L_{S_pz}.
$$
Note now that $\U_{\varepsilon_p}=L_{\varepsilon_p^2}=1$ again from \cite[Theorem 5]{youngson2}.  Thus
\begin{align*}
\U_{S_p z}=2 L_{\varepsilon_p}^2(L_z)^2-L_{z^2}=2 L_1(L_z)^2-L_{z^2}=\U_z=S_pS_p^{-1}\U_z 1=S_p\U_zS_p^{-1} \U_{\varepsilon_p}
\end{align*}
since $\U_z$ commutes with $S_p$ (which is its own inverse). 
\end{proof}

\begin{rem}[Each central projection gives a different copy of the cone] For a central projection $p\in\V$, the set $S_p(\Omega)$ will then be a cone; for $p=1$ we obtain $S_p=Id$ and the cone $\Omega$ while for $p=0$ we obtain $S_p=-Id$ and the cone $-\Omega$. We claim that \textit{different central projections give different cones}: if $S_{p_1}\GO=S_{p_2}\GO$, then $S_{p_1}^{-1}S_{p_2} = S_{p_1}S_{p_2}=L_{\varepsilon_{p_1}}L_{\varepsilon_{p_2}}\in \GO$. Let $p=(\varepsilon_{p_1}\varepsilon_{p_2}+1)/2$; from 
$$
(\varepsilon_{p_1}\varepsilon_{p_2})^2=\varepsilon_{p_1}^2\varepsilon_{p_2}^2=1\cdot 1=1
$$
(since $L_{p_1}$ commutes with $L_{p_2}$), we see that $p^2=p$ is a central projection. Then $S_p=L_{\varepsilon_p}=L_{\varepsilon_{p_1}}L_{\varepsilon_{p_2}}\in \GO$, and in particular $2p-1=S_p(1)>0$. Thus $p>1/2$ and in particular $p$ is invertible. Since $p^2=p$, it must be then that $p=1$, thus $S_p = S_{p_1}S_{p_2} =Id$. We conclude that $S_{p_1}=S_{p_2}$ or equivalently that $p_1=p_2$.
\end{rem}

\begin{rem}\label{remecucentproy}Let $p\in \V$ be an idempotent, $p^2=p$, let $p'=1-p$. Then $\U_{\varepsilon_p}$ is an involutive automorphism, and the following identities are elementary:
$$
\begin{array}{ll}
a)\; \U_{\varepsilon_p}=8L_p^2-8L_p+1=1-4\U_{p,p'} &\qquad b)\; \U_p-\U_{p'}=2L_p-1=L_{\varepsilon_p}\\
c)\; 2\U_{p,p'}=4L_p(1-L_p) &\qquad d)\; L_{\varepsilon_p}L_p(1-L_p)=0\\
e)\; L_{\varepsilon_p}U_{p,p'}=0 & \qquad f)\; L_{\varepsilon_p}^2=1-2\U_{p,p'}.
\end{array}
$$
The last identity follows from the fact that $\sigma(L_p)\subset \{0,\frac{1}{2},1\}$. 
\end{rem}

\begin{rem}\label{rempiercecentr} As we have discussed before in Theorem \ref{upositive}, the space $\V$ can be decomposed into three summands $J_0$, $J_1$ and $J_2$ by means of the Pierce decompostion, where each of these spaces is the range of the projections $U_p$, $U_{p'}$ and $2U_{p,p'}$ respectively. But these spaces are also the eigenspaces of the operator $L_p$, associated to the eigenvalues $0$, $1/2$ and $1$. In this section we will rename these $J$ spaces as $\V_1^p$, $\V_0^p$ and $\V_{1/2}^p$. For more details and proof of the assertions used see \cite[Theorems 8.1.4 and 8.2.1]{mccrimmon}.  Let's regroup the summands of the Pierce decomposition by means of the involutive automorphism $\U_{\varepsilon_p}$ as follows:
$$
\V^p=\V^p_0\oplus \V_1^p = \{v\in\V: \U_{\varepsilon_p}v=v\}, \qquad \V^p_{1/2}=\{v\in\V: \U_{\varepsilon_p}v=-v\}=\ker(L_{\varepsilon_p}).
$$
Then $\V=\V^p\oplus \V^p_{1/2}$; the fact that the sum is direct is reflected in the fact that $L_{\varepsilon_p}\U_{p,p'}=0$. We have that $\V^p=\ker(\U_{p,p'})=\ker (L_p^2-L_p)$ is a JB-subalgebra of $\V$. Moreover, since $L_p^2=L_p$ in $\V^p$, then $p$ is a central projection in $\V_p$ and thus $L_{\varepsilon_p}^2$ is the projection onto $\V^p$ and it is the identity there (Remarks \ref{xvslx} and \ref{unoth}).

\smallskip

 On the other hand $\V^p_{1/2}=\ker(2L_p-1)$ is a subspace, and $2\U_{p,p'}$ is an idempotent onto it,  but it is not a subalgebra: it contains no squares. However if $z\in \V^p$ or $z\in \V^p_{1/2}$ it is plain that $U_z(\V^p_{1/2})\subset \V^p_{1/2}$ and $\V^p_{1/2}$ is a Jordan triple system. Let $x\in \V^p$ and $y\in \V^p_{1/2}$ and it can be checked by hand that
%\begin{enumerate}[label=\roman*)]
$$
\U_px=px,\quad\U_py=0,\quad L_{\varepsilon_p}^2x=x, \quad  L_{\varepsilon_p}y=0.
$$
%\item $\U_p\U_y\U_p=\U_{\U_py}=\U_0=0$
%\item $L_p\U_yL_px=0$ and $L_{\varepsilon_p}\U_yx=-\U_yL_{\varepsilon_p}x$, thus $L_{\varepsilon_p}\U_yL_{\varepsilon_p}x=-U_yx$.
%\end{enumerate}
\end{rem}

We now give a full characterization of the elements in $\Str$:

\begin{teo}\label{stru}Let $g\in \Str$, then there exist $v\in\Omega$, a central projection $p\in \V$ and an automorphism $k\in\Aut$ such that
$$
g=\U_v S_p k=S_p\U_v k.
$$
\end{teo}
\begin{proof}
Let $z=g(1)$, from Remark \ref{ginvertible} we know that $z$ is invertible, and by Remark \ref{uequis} we can write $z=z_+-z_-$ with $z_{\pm}\ge  0$. Let $p,p'$ be range projections for $z_{\pm}$ respectively; since $z$ is invertible it must be $p'=1-p$. Let $\varepsilon_p=2p-1$, and consider  $|z|=z_++z_-=\varepsilon_p z$, then $|z|\in \Omega$. Let $v=\sqrt{|z|}\in\Omega$, and let $h=\U_v^{-1}g\in \Str$. Now it only suffices to show that $h = S_p k$, as $g = U_vh$. We note that 
$$
h(1)=\U_v^{-1}z=\U_v^{-1}\varepsilon_p v^2=\varepsilon_p
$$
since $v,\varepsilon_p\in \mathcal C(z)$. Moreover, if we consider $h^{-1}(1)$, then $Id = \U_{h^{-1}h1} = h^{-1}\U_{\varepsilon_p}h\U_{h^{-1}1}$ and therefore $\U_{h^{-1}1}=h^{-1}\U_{\varepsilon_p}h$. Hence $\U_{(h^{-1}(1))^2} = \U_{h^{-1}(1)}^2=h^{-1}\U_{\varepsilon_p}^2h=Id$. Thus, by Theorem \ref{upositive} as $(h^{-1}(1))^2$ is positive it must be $(h^{-1}(1))^2=1$, and then $h^{-1}(1)=\varepsilon_q$ for some idempotent $q\in \V$. Let $\V^{\mathbb C}=\V\oplus i\V$ be the complexification of $\V$ making it a $JB^*$-algebra; since $h\in \Str$, by Remark \ref{gcestr} we know that $h^{\mathbb C}$ (the complexification of $h$) belongs to $\texttt{Str}(V^{\mathbb C})$. We will call this complexification $h$, for short.  It is plain that $p+ip'\in \V^{\mathbb C}$, is an element in (the complexification of) $\mathcal C(p)$, and it is a square root of the symmetry $\varepsilon_p$, as $(p+ip')^2=p-p'+0=\varepsilon_p$. Let $k=\U_{p+ip'}h$, then $k\in \texttt{Str}(V^{\mathbb C})$ and
$$
k(1)=\U_{p+ip'}\varepsilon_p=(p+ip')^2\varepsilon_p=\varepsilon_p^2=1,
$$
where this is valid as the operations happen inside $\mathcal C(p)$. Thus it must be that $k\in \texttt{Aut}(V^{\mathbb C})$ (see Lemma \ref{auto1}). Since $(p+ip')^{-1}=p-ip'$, we have $h=\U_{p-ip'}k$. Note that
$$
k(\varepsilon_q) = U_{p+ip'}h(\varepsilon_q) = U_{p+ip'}(1) = (p+ip')^2 = \varepsilon_p,
$$
hence $k(q)=p$ and therefore $k(q')=p'$ and $h=k\U_{q-iq'}=\U_{p-ip'}k$ as $k$ is an automorphism. We claim that $hL_q=L_ph$ in $\V^{\mathbb C}$: from $h\U_{q+iq'}=k=\U_{p+ip'}h$ and the fact that $\U_{p+ip'}= U_p - U_{p'} + 2i\U_{p,p'} = L_{\varepsilon_p}+2i\U_{p,p'}$ (and likewise for $q$), we get
$$
hL_{\varepsilon_q}+2i h\U_{q,q'}=k=L_{\varepsilon_p}h+2i\U_{p,p'}h,
$$
and evaluating in $x\in \V$ it must be that $hL_{\varepsilon_q}(x)=L_{\varepsilon_p}h(x)$, since $h$ maps $\V$ into itself. Thus $hL_q=L_ph$ in $\V$, but passing to the complexification it is plain that $hL_q=L_qh$ holds also in $\V^{\mathbb C}$. It is obvious that the same applies to $k$, $kL_q=L_pk$, as $k$ is an automorphism. Now we write $k(v)=\alpha(v)+i\beta(v)$, with $\mathbb{R}-$linear $\alpha,\beta:\V^{\mathbb C}\to \V$, the real and imaginary parts of $k$ given by
$$
\alpha(v)=\nicefrac{1}{2}(k(v)+k(v)^*),\qquad  \beta(v)=\nicefrac{1}{2i}(k(v)-k(v)^*).
$$
We compute
\begin{align*}
h&=\U_{p-ip'}k=(L_{\varepsilon_p}-2i\U_{p,p'})(\alpha+i\beta)=L_{\varepsilon_p}\alpha+2\U_{p,p'}\beta+i[L_{\varepsilon_p}\beta-2 \U_{p,p'}\alpha].
\end{align*}
Since $h$ maps $\V$ into $\V$, for $x\in \V$ it must be
$$
L_{\varepsilon_p}\beta(x) -2 \U_{p,p'}\alpha(x)=0.
$$
Applying $L_{\varepsilon_p}$, we see that $L_{\varepsilon_p}^2\beta(x)=0$, thus $\beta(x)=2\U_{p,p'}\beta(x)=4L_p(1-L_p)\beta(x)$ by Remark \ref{remecucentproy}. Appling $L_{\varepsilon_p}$ again, we see that $L_{\varepsilon_p}\beta(x)=0$, or equivalently, that $\beta(x)\in \V^p_{1/2}$ when $x\in\V$. Then it must also be that $2\U_{p,p'}\alpha(x)=0$, and this in turn implies that $\alpha(x)\in \V^p$ when $x\in \V$. Hence, for $x\in \V$ we have
$$
h(x)=L_{\varepsilon_p}\alpha(x)+2\U_{p,p'}\beta(x)=L_{\varepsilon_p}\alpha(x)+\beta(x).
$$
We claim that $L_p\alpha=\alpha L_q$. To prove it, first we right-multiply $L_{\varepsilon_q}$ in the previous identity, and by Remarks \ref{remecucentproy} and \ref{rempiercecentr} we have
\begin{equation}\label{alfa}
L_{\varepsilon_p}\alpha(L_{\varepsilon_q}x)+\beta(L_{\varepsilon_q}x)=h(L_{\varepsilon_q}x)=L_{\varepsilon_p}h(x)=L_{\varepsilon_p}^2\alpha(x)+0=\alpha(x).
\end{equation}
Applying again $L_{\varepsilon_p}$ we see that $\alpha L_{\varepsilon_q}=L_{\varepsilon_p}\alpha$. On the other hand, from this, the fact that $L_{\varepsilon_p}^2\alpha(x)=\alpha(x)$ for $x\in\V$ and equation (\ref{alfa}) we see that
$$
\alpha(x)=L_{\varepsilon_p}\alpha(L_{\varepsilon_q}x)+\beta(L_{\varepsilon_q}x)=\alpha(x)+\beta(L_{\varepsilon_q}x),
$$
hence it must be $\beta L_{\varepsilon_q}=0$. From $L_{\varepsilon_p}\beta = 0$ it follows that $2L_p\beta=\beta$ and analogously that $2\beta L_q = \beta$, then we have $\beta L_q=\beta/2=L_p\beta$ also.  Let us now show that $\ker(\alpha|_{\V})=\V^q_{1/2}$ and that $\ker(\beta|_{\V})=\V^q$. We will call these restrictions $\alpha$ and $\beta$ for short. First we write
$$
kL_{\varepsilon_q}(y)=L_{\varepsilon_p}k(y)=L_{\varepsilon_p}\alpha(y) +i L_{\varepsilon_p}\beta(y) = L_{\varepsilon_p}\alpha(y).
$$
Then if $\alpha(y)=0$, it must be that $L_{\varepsilon_q}y=0$ thus $y \in V^q_{1/2}$. Reciprocally, if $y \in V^q_{1/2}$ then $L_{\varepsilon_q}y=0$, which in turn implies $L_{\varepsilon_p}\alpha(y)=0$, and applying $L_{\varepsilon_p}$ we see that $y\in \ker (\alpha)$. Now for the kernel of $\beta$, we write 
$$
k(2\U_{q,q'}x)=2\U_{p,p'}k(x)=2\U_{p,p'}\alpha(x) + 2i\U_{p,p'}\beta(x)=i2\U_{p,p'}\beta(x)=i\beta(x).
$$
If $x\in\ker\beta$, then $U_{q,q'}x=0$ thus $x\in \V^q$. Reciprocally, if $x\in\V^q$ then $\U_{q,q'}x=0$ thus $k(2\U_{q,q'}x) =i\beta(x)=0$, thus $x\in \ker\beta$.  Assumme that there exist $y\ne 0$ in $\V^q_{\nicefrac{1}{2}}$. Then $y^2>0$ and moreover it is plain that $y^2\in \V^q$ by Remark \ref{rempiercecentr}. Thus there exists $0\ne x\in \V^q$ such that $x^2=y^2$. Since $k(x)=\alpha(x)+i\beta(x)=\alpha(x)$, and $\alpha$ is nonzero in $\V^q\setminus\{0\}$, we get
$$
k(x^2)=k(x)^2=\alpha(x)^2>0.
$$
On the other hand, since $\alpha(y)=0$ and $\beta$ is nonzero in $\V_{\nicefrac{1}{2}}^q\setminus \{0\}$, we have 
$$
k(x^2)=k(y^2)=k(y)^2= (i\beta(y))^2=-\beta(y)^2<0,
$$
a contradiction. Thus it must be $\V^q_{1/2}=\{0\}$ and $\V=\V^q$, so $\beta$ is null in $\V$. Then $k|_{\V} = \alpha|_{\V}$ thus $k$ maps $\V$ into $\V$.  Moreover, since $\U_{\varepsilon_q}\equiv 1$ in $\V$, it follows that $q$ is a central projection. Then for all $z$
$$
L_pL_zk = kL_qL_{k^{-1}(z)} = kL_{k^{-1}(z)}L_q = L_zL_pk,
$$
and thus $p$ is also central. Then,
$$
h = U_{p-ip'}k = (U_p - U_{p'} + 2iU_{p,p'})k = L_{\varepsilon_p}k
$$
and thus $g = U_vh = U_vL_{\varepsilon_p}k$.
%In the process, we note that we obtained
%$$\alpha=L_{\varepsilon_p}kL_{\varepsilon_q}\qquad \textrm{ and }\qquad \beta=-2ik\U_{q,q'}=-2i\U_{p,p'}k.$$
%From here it is easy to see that the range of $\alpha$ is exactly $\V^p$ and that the range of $\beta$ is exactly $\V^q_{1/2}$. Then it is plain that $\alpha^{-1}|_{\V^p}=k^{-1}|_{\V^p}$ and that $\beta^{-1}|_{\V^p_{1/2}}=-ik^{-1}|_{\V^p_{1/2}}$. Then if we write $x+y\in \V=\V^q\oplus \V^q_{1/2}$, we have
%\begin{equation}\label{hh}
%h(x+y)=L_{\varepsilon_p}\alpha(x)+\beta(y)\in \V^p\oplus \V^p_{1/2},
%\end{equation}
%and this finishes the proof.
\end{proof}

\begin{coro}\label{GOcopiesStr}
There exist one different copy of $\GO$ in $\Str$ for each central projection $p\in \V$ (given by $S_p\GO$), and all copies are obtained in such fashion.
\end{coro}

We can now prove Theorem \ref{caractautomorf}:

\begin{teo}
	Let $g\in \Str$. Then $g^* = g^{-1}$ if and only if $g =  S_pk = L_{\varepsilon_p} k $, where $k$ is a multiplicative automorphism of $\V$ and $\varepsilon_p$ is a central symmetry.
\end{teo}
\begin{proof}
If $g  = S_p k $ then by Remark \ref{rem:autinv} we have $k^*= k^{-1}$ and $S_p^* = S_p^{-1} U_{S_p1} = S_p^{-1} U_{\varepsilon_p} = S_p^{-1}$, so $g^* = g^{-1}$. Now assumme $g^* = g^{-1}$, we know by Theorem \ref{stru} that $g = \U_xS_pk$ with positive $x$; replacing this in the equality we obtain that
	\begin{equation*}
		k^{-1}S_p^{-1}\U_x = k^{-1}S_p^{-1}\U_x^{-1}
	\end{equation*}
	and therefore $\U_x^2 = \U_{x^2} = Id$. Then $x^4 = 1$, but as both $x^2$ and $x$ are positive and there is an unique positive square root we have that $x = 1$. This gives us $g = S_pk$.
\end{proof}

There is a better description of $g(\Omega)$: it is the cone of positive elements for a different product in $\V$.

\begin{defi}[Jordan homotopes and isotopes]
	Let $u$ be an element of $\V$. Define a new product in $\V$ as $x \cdot_u y = \U_{x,y}(u)$. This new product induces a new Jordan algebra $\V_u$, not necessarily isomorphic to the original, called a Jordan homotope. Moreover, $\V_u$ will be unital if and only if $u$ is an invertible element, and $1_u = u^{-1}$. In this case, $V_u$ is called an Jordan isotope.
\end{defi}

We will asume that $u$ is invertible, so $\V_u$ is an unital algebra. We will denote $x^{2^{u}}= x \cdot_u x$; $x^{-1^{u}}$ the inverse for the $u$-product; and $\U_x^{u}$ and $\U_{x,y}^{u}$ the quadratic and bilinear operators for the $u$-product.

\begin{rem}
	With this new product we can define the usual elements and operations of a Jordan algebra. For example,
	\begin{align*}
		& x^{2^{u}}= \U_xu, & \U_x^{u} = \U_x\U_u, && \U_{x,y}^{u}=\U_{x,y}\U_u.
	\end{align*}
	An element $x$ is invertible with the new product if and only if it is invertible with the original product, and $x^{-1^{u}} = \U_u^{-1}x^{-1}$. 

	\smallskip

	One can see that for every $g\in \Str$, the map $g$ is a Jordan (i.e. multiplicative) isomorphism between $\V$ and $\V_{g(1)^{-1}}$. Moreover, this property characterizes the structure group. To expand on these topics, see the exercises at the end of \cite[Part II, Chapter 7]{mccrimmon}.
\end{rem}

\begin{prop}
	Let $g\in \Str$. Then $g(\Omega)$ is $\Omega^{g(1)^{-1}}$, the cone of positive elements of $\V_{g(1)^{-1}}$.
\end{prop}
\begin{proof}
	Let $z$ be an element in $\V$. We want to see that $z$ belongs to $\Omega$ if and only if $g(z)$ is positive in $\V_{g(1)^{-1}}$.  Let $\lambda$ be a real number, then $g(z) - \lambda g(1) = g(z - \lambda1)$. As $g$ belongs to $\Str$, $g(z - \lambda1)$ is invertible if and only if $z -\lambda1$ is invertible. This tells us that the spectrum of $z$ in $\V$ is the same as the spectrum of $g(z)$ in $\V_{g(1)^{-1}}$, as $1_{g(1)^{-1}} = g(1)$. Then, $z$ is positive in $\V$ if and only if $g(z)$ is positive in $\V_{g(1)^{-1}}$.
\end{proof}

As we have seen before, two elements $g$ and $h$ in the same coclass give us the same cone $g(\Omega) = h(\Omega)$. Together with the last result, this tells us that the notion of positivity in different $g(1)^{-1}$-products for each $g$ in the structure group does not vary inside the coclass. Then 

\begin{coro} There exist as many copies of $\GO$ in $\Str$, as distinctive cones of positive elements $\Omega^{g(1)^{-1}} = g(\Omega)$ for different $g(1)^{-1}$-products.
\end{coro}

Moreover, we have that 
\begin{coro}
	The isotope $\V_u$ is Jordan isomorphic to $\V$ if and only if there exist $v\in \Omega$ and a  central symmetry $\varepsilon_p$ such that $u = \U_v\varepsilon_p = v^2\varepsilon_p$.
\end{coro}
\begin{proof}
	Let $g: \V \to \V_u$ be the isomorphism. As we said before, if we consider $g: \V \to \V$ then $g$ belongs to $\Str$, so by Theorem \ref{stru} there exist $v\in \Omega$, a central symmetry $\varepsilon_p$ and an automorphism $k$ such that $g= \U_vS_pk$, and $g(1) = v^2\varepsilon_p$. As $g$ is also a Jordan isomorphism between $\V$ and $\V_{g(1)^{-1}}$, we have that $Id:\V_u \to \V_{g(1)^{-1}}$ must also be a multiplicative isomorphism, and the unit must be the same. So $u^{-1} = g(1)$, and $u = (v^{-1})^2\varepsilon_p$. Now, let $u = v^2\varepsilon_p$ for a positive $v$ and central symmetry $\varepsilon_p$. We know by Lemma \ref{LepsIsStr} that $L_{\varepsilon_p}$ belongs to $\Str$ and so does $g = \U_{v^{-1}}L_{\varepsilon_p}$ with $g(1)^{-1} = u$. Then $g: \V \to \V_{u}$ is a multiplicative isomorphism.
\end{proof}

We can give another characterization of the isotopes of $\V$ by the $\U$ operators.
\begin{coro}
	The isotope $\V_x$ is Jordan isomorphic to $\V$ if and only if $\U_x$ is a positive operator.
\end{coro}
\begin{proof}
We have seen in \ref{upositive} that $\U_x$ is positive if and only if $x = v\varepsilon_p$ with $v$ positive and $\varepsilon_p$ a central symmetry, this together with the last corollary gives us the result.
\end{proof}

We have given a characterization of the isomorphic isotopes of $\V$: there is one isomorphic isotope for each element inside one of the cones of the family $\{g(\Omega) : g\in  \Str\}$. Then for associative algebras endowed with the Jordan product this gives an indication of the number of cones in $\V$:

\begin{coro}
	Let $\V$ be an associative algebra endowed with the Jordan product. Then 
	\begin{equation*}
		\glv = \cup_{g \in \Str} g(\Omega)=\{g(1):g\in \Str\}.
	\end{equation*}
\end{coro}
\begin{proof}
	By \cite[Theorem II.7.5.1]{mccrimmon}, we have that for every $u$ invertible $\V_u$ is isomorphic to $\V$. Then every invertible element is in one of the cones $g(\Omega)$ for $g \in \Str$.
\end{proof}
Note that this result is not valid for every special Jordan algebra: if we have a Jordan subalgebra that is not a subalgebra in the original product, it does not apply.

\section{The special Jordan algebra of Hilbert space operators}\label{s4}

Consider $\V = \bh_{sa}$, the self-adjoint operators of a complex separable Hilbert space $\h$ with the Jordan product  $A\circ B=1/2(AB+BA)$, and the spectral norm as JB-algebra norm. Then  $\Omega\subset \V$  is the set of positive invertible  operators, and we use $A>0$ to denote $A\in \Omega$. 

\medskip

We will characterize $\GO,\Aut$ and $\Str$. In order to achieve this, we need to introduce a few notions related to antilinear operators. We will reserve the star $*$ for the adjoint in the structure group, and we will use a dagger $\dagger$ to indicate the usual Hilbert space adjoint.

\begin{defi}A map $f:\h\to\h$ is called \textit{antilinear} if $f(x+y)=f(x) + f(y)$ and $f(\lambda x) = \overline{\lambda}f(x)$ for every $x,y$ in $\h$ and complex number $\lambda$. Let $\langle\cdot,\cdot\rangle$ denote the Hilbert space inner product, which we assumme is conjugate linear in the second variable. We can define another $\dagger$ operation (called the \textit{conjugate adjoint}): given $f$, and by means of the Riesz representation theorem for $\h$, we let $f^{\dagger}$ be the unique antilinear operator such that
$\langle fx,y \rangle = \overline{\langle x,f^{\dagger}y \rangle}$	for every $x,y\in \h$. From the context it will be apparent if the $\dagger$ denotes the usual adjoint or the conjugate adjoint.		 It is easy to check that the composition of two antilinear operators is linear. 	An \textit{antiunitary} operator is an antilinear operator $U:\h\to \h$ such that
	\begin{equation*}
		\langle Ux,Uy \rangle = \overline{\langle x,y \rangle}
	\end{equation*}
	for every $x,y\in \h$, or, equivalently, such that $U^{-1} = U^{\dagger}$. The product of two antiunitary operators is unitary, and $U$ is antiunitary if and only if $U$ is antilinear and $\|\U\xi\|=\|\xi\|$ for all $\xi\in \h$. 
\end{defi}
It is also easy to check that, as linear operators,  antilinear operators on a Hilbert space have a (right) polar decomposition: every antilinear map $f$ can be written as $f = |f^{\dagger}|U$, where $|f^{\dagger}| = \sqrt{ff^{\dagger}}$, and $U$ is a partial antilinear isometry. $|f|$ is a positive linear operator and in case that $f$ is invertible, $U$ is antiunitary.  A good reference on the subject of antilinear and antiunitary operators is the paper by Routsalainen \cite{rout}.

\smallskip

An antiunitary operator $J$ is a \textit{conjugation} if $J^2=1$. The typical example is given by fixing an orthonormal basis $\{e_i\}_{i\in\mathbb N}$ of $\h$, and defining $J(\sum \alpha_ie_i)=\sum \overline{\alpha_i}e_i$, we call this \textit{conjugation in a basis}. The following characterizations that can be found in  \cite[pag. 194]{rout} will be useful

\begin{prop}Let $J:\h\to\h$ be antilinear and consider the conditions $J=J^{\dagger}$, $J^{\dagger}=J^{-1}$, $J^2=1$. Then any two conditions imply the third, and $J$ is a conjugation. For any conjugation $J$ there exists a basis such that $J$ is conjugation in that basis.
\end{prop}

\subsection{The group $G(\Omega)$ and its components}

Now we can charaterize the group preserving the positive cone:

\begin{teo}\label{gomegabh}
Take $\V=\bh_{sa}$ as a JB-algebra. Then 
\begin{enumerate}
\item $g\in \GO$ if and only if there exists $f:\h\to \h$ invertible linear (or antilinear) such that $g(A)=fAf^{\dagger}$, where the dagger $\dagger$  denotes the usual adjoint (resp. the conjugate adjoint). 
\item  If $g(A)=fAf^{\dagger}$, then $g^*(A)=f^{\dagger} Af$.
\item  Assumme $g=f\cdot f^{\dagger}=h\cdot h^{\dagger}$. Then both $f,h$ are linear or both are antilinear, and  there exists $\lambda\in S^1$ such that $f=\lambda h$. 
\item The decomposition $g=\U_x k$ with $x\in \Omega$ and $k\in \Aut$ is given by $x=|f^{\dagger}|$ and $k(A)=UAU^{\dagger}$ for some unitary or antiunitary operaror $U$ in $\h$, i.e. if $f=|f^{\dagger}|U$ is the polar decomposition of $f$, then $g(A)=|f^{\dagger}|UAU^{\dagger}|f^{\dagger}|$.
\item $g=f\cdot f^{\dagger}\in \Aut$ if and only if $f$ is unitary or antiunatary.
\item The set of linear and the set of antilinear maps induce the two connected components of $\GO$, and the set of unitary and antiunitary maps induce the two connected components of $\Aut$.
\end{enumerate}
\end{teo}
\begin{proof}
Clearly each such map $A\mapsto fAf^{\dagger}$ is bounded and real linear, and preserves $\Omega$: if $A>0$ then $fAf^{\dagger}>0$. This is trivial if $f$ is linear but also true if $f$ is antilinear, since 
$$
\langle fAf^{\dagger}\xi,\xi\rangle= \overline{\langle A^{1/2} f^{\dagger}\xi , A^{1/2} f^{\dagger}\xi\rangle}=\|A^{1/2} f^{\dagger}\xi\|^2>0.
$$ 
Now assumme that $g\in \GO$, then $g=\U_Xk$ for some $X>0$ and $k$ an automorphism of $\V$. Then $k\U_Ak^{-1}=\U_{kA}$ for any $A\in\V$, thus 
$$
k(ABA)=k(A)k(B)k(A)\qquad \forall\, A,B\in \V.
$$
It is well-known that then there exists a unitary or antiunitary operator $U$ such that $k(A)=UAU^{\dagger}$ for any $A\in\V$ (see  for instance \cite[Theorem 3.2]{anhou} for a proof), where $\dagger$ is the usual adjoint or the conjugate adjoint. Thus $g(A)=X UAU^{\dagger} X=fAf^{\dagger}$ if we let $f=XU$. This proves the first assertion. Now recall that $g^*=g^{-1}\U_{g1}$ thus in this case we have $g(1)=ff^{\dagger}=|f^{\dagger}|^2$ and then
$$
g^*(A)=g^{-1}(ff^{\dagger} Aff^{\dagger})=f^{-1}ff^{\dagger}Aff^{\dagger}(f^{-1})^{\dagger}=f^{\dagger} Af.
$$
Assumme now that $fAf^{\dagger}=h Ah^{\dagger}$ for every $A\in \bh$, taking $A=1$ we see that $|f^{\dagger}|=|h^{\dagger}|$, then by polar decomposition we have $|f^{\dagger}|U AU^{\dagger}|f^{\dagger}|=|f^{\dagger}| W AW^{\dagger}|f^{\dagger}|$ thus $UAU^{\dagger}=WAW^{\dagger}$ for any $A\in \bh_{sa}$, equivalently $W^{-1}U A=A W^{-1}U$. If $f$ is linear and $h$ is antilinear (or viceversa) then $T=W^{-1}U$ is antiunitary, and $TA=AT$ for all $A\in \bh_{sa}$. In particular for any $\xi\in \h$ we have
\begin{equation}\label{txi}
\|\xi\|^2 T\xi=T \langle \xi,\xi\rangle\xi =T(\xi\otimes\xi)\xi=(\xi\otimes\xi)T\xi=\langle T\xi,\xi\rangle\xi,
\end{equation}
and since $T$ is an isometry,  $\|\xi\|^2 \|\xi\|= |\langle T\xi,\xi\rangle| \|\xi\|$, that is $|\langle T\xi,\xi\rangle|=\|\xi\|^2$ for all $\xi\in\h$. If we apply $T$ on both sides of (\ref{txi}) and then multiply by $\|\xi\|^2$, we obtain
$$
 \|\xi\|^2 \|\xi\|^2 T^2\xi=\overline{\langle T\xi,\xi\rangle}\|\xi\|^2T\xi=|\langle T\xi,\xi\rangle|^2 \xi= \|\xi\|^4 \xi.
$$
Hence $T^2=1$ and by the previous proposition, $T$ is a conjugation, in particular a conjugation in a basis $\{e_i\}_i$. Let $A=i e_1\otimes e_2-i e_2\otimes e_1\in \V$, then $ATe_1=Ae_1=ie_2$ and on the other hand $TAe_1=T(-ie_2)=-e_1\ne ATe_1$, a contradition.  Thus it must be that both $f,h$ are linear or both are antilinear, then $U,W$ are both unitary or both antiunitary, thus $W^{-1}U$ is unitary. We see that $W^{-1}U$ is in the center of $\bh$, which is $\mathbb C 1$, hence $W=e^{i\theta} U$, thus $f=e^{i\theta}h$.

The fourth and fifth assertions are apparent from the previous discussions. Now note that if $k=U\cdot U^{\dagger}$ with $U$ a unitary operator in $\bh$, then there exists skew-adjoint $Z\in \bh$ such that $U=e^Z$. Let $k_t=e^{tZ}\cdot e^{-tZ}$. Then $k_t\subset \Aut$, $k_0=id$ and $k_1=k$, thus $k\in \Aut_0$ (all unitaries belong to the same component). Assumme that $k_n=U_n \cdot U_n^*\to_n k=U \cdot U^*$. In particular $U_n(\xi\otimes\xi)U_n^{\dagger}\to U(\xi\otimes \xi)U^{\dagger}$ for each $\xi\in \h$, thus $(U_n\xi)\otimes (U_n\xi)\to (U\xi)\otimes (U\xi)$ and this is only possible (since $U_n,U$ are isometries) if there exists $\theta(n,\xi)\in [0,2\pi]$ such that  $e^{i\theta(n,\xi)}U_n\xi\to U\xi$. In particular, if all the $U_n$ are linear, then $U$ must be linear, and if all the $U_n$ are antilinear, $U$ must be antilinear. 

Now let $k_t:[0,1]\to  \Aut$ be a continuous path joining $Id$ with  $g= U \cdot U^{\dagger}$ for an antiunitary operator $U$. Take $t_0 = \inf \{t: U_t\text{ is antilinear}\}$, note that $t_0>0$. Let $U_n=U_{t_0-1/n}$, then from the continuity of $k_t$ and the previous discussion, we see that $U_{t_0}$ is linear, thus unitary. But if we take $U_n=U_{t_0+1/n}$ we see that $U_{t_0}$ is also antiunitary, and this is impossible. Thus there is no such path and the antiunitaries belong to a different component than the unitaries. On the other hand if $U$ is antiunitary and $J$ is complex conjugation in a fixed orthonormal basis, then $JU$ is unitary thus $JU=e^{Z}$ for some linear $Z^{\dagger}=-Z$, and we can join $k=U\cdot U^{\dagger}$ with $j=J\cdot J$ with a continuous path. In particular all antiunitaries belong to the same component, showing that the sets of unitaries and antiunitaries induce the two components of $\Aut$. The assertion for $\GO$ is now apparent from Theorem \ref{retract}, previous remarks and the polar decomposition.
\end{proof}

\begin{rem}The main tool used in the previous characterization is the theorem stating that a map preserving the triple product, must be of the prescribed form. However, it is not clear how does $U$ varies as $k=U\cdot U^{\dagger}$ varies, for instance if $t\mapsto k_t$ is a continuous (or smooth) map into $\Aut$, and we represent $k_t=U_t\cdot U_t^{\dagger}$, is the map $t\mapsto U_t$ continuous (or smooth)? Since there is some ambiguity (the factor $\lambda\in S^1$), can we pick $U$ adequately so that it is well-behaved? Next we show that it is possible, using a well-known trick that exhibits the unitary.
\end{rem}

\medskip

First we need a quick remark: note that if $p^2=p=p^{\dagger}\in\bh$ then $\varepsilon_p=2p-1$ is a symmetry, i.e. $\varepsilon_p=\varepsilon_p^{-1}=\varepsilon_p^{\dagger}$; in particular $\varepsilon_p$ is unitary and if $q$ is another projection with $\|q-p\|_{\infty}<1$ then 
$$
\|\varepsilon_p \varepsilon_q-1\|_{\infty}=\|\varepsilon_p-\varepsilon_q\|_{\infty}=2\|p-q\|<2
$$
thus $\varepsilon_p\varepsilon_q$ has an analytic logarithm $Z$ in $\bh$, which is skew-adjoint and depends smoothly on $q$; moreover it in not hard to see that $\varepsilon_pe^{Z}= e^{-Z}\varepsilon_p$ since $Z$ is $p$-codiagonal. Recall also that the unitary group $\mathcal U(\h)$ is a Banach-Lie embedded subgroup of $\bh$ with the uniform norm.

\begin{teo}\label{levantada}
Let $k\in \Aut$. Fix a unit norm $\xi\in \h$, let $p=\xi\otimes \xi$ be its one-dimensional projection.
\begin{enumerate}
\item Assumme that $\|k-1\|<1$. Let $Z$ be the linear skew-adjoint operator given by
$$
Z(k)=1/2\ln(\varepsilon_{k(p)}\varepsilon_p)=1/2\ln((2k(p)-1)(2p-1)).
$$
Then $k=e^{Z(k)}W(k)\,\cdot \,W(k)^{\dagger}e^{-Z(k)}$ with unitary $W(k)$, where for each $\eta\in\h$ 
$$
W(k)\eta =e^{-\ad Z(k)}k^{\mathbb C}(\eta\otimes\xi)\xi=e^{-Z(k)}k^{\mathbb C}(\eta\otimes\xi)e^{Z(k)}\xi
$$
(here $k^{\mathbb C}$ is the complexification  $k(A+iB)=kA+ikB$ for $A,B\in\V$).
\item Let $J$ be a conjugation and $j=J\cdot J$ such that $\|k-j\|<1$, then $k=Je^{Z(k)}W(k)\cdot W(k)^{\dagger}e^{-Z(k)}J$ for the same maps $Z,W$ as above. 
\item The map  $s:\{k\in \Aut: \|k-1\|<1\}\to \mathcal U(\h)$ given by $s:k\mapsto e^{Z(k)}W(k)$ is smooth, moreover it is real analytic.
\end{enumerate}
\end{teo}
\begin{proof}
%If $q\le k(p)$ is a projection, then $k(p)=q+(k(p)-q)$ is a sum of two projections, thus $p=q_0+(p-q_0)$ with $q_0=k^{-1}q$. But since $p$ is a minimal projection, one of these must be zero, but then $q=0$ or $q=k(p)$, thus $k(p)$ is a one-dimensional projection. 
Since $k(p^2)=k(p)^2$ we see that $k(p)$ is an orthogonal projection. Now $\|k(p)-p\|_{\infty}\le \|k-1\| \,\|p\|_{\infty}<1$, thus taking $Z$ as described gives a linear skew-adjoint operator depending smoothly on $k$ such that $e^{Z(k)}pe^{-Z(k)}=k(p)$ (see for instance \cite[Proposition 3.1]{ancomb}). Let  
$$
\lambda_k=e^{-\ad Z(k)}k^{\mathbb C},
$$
then $\lambda_k(p)=p$, and since $k^{\mathbb C}=U_k\,\cdot\, U_k^{\dagger}$ for some unitary operator by the previous theorem, we see that $\lambda_k(AB)=\lambda_k(A) \lambda_k(B)$ for any $A,B\in\bh$. Let $W(k)$ be as in the formula above, then it is clear that it is bounded linear. Let's see that $W(k)$ is unitary: first note that if we put 
$$
V(k)\eta=\lambda_k^{-1}(\eta\otimes \xi)\xi
$$
it is easy to check that $V(k)$ is the left and right inverse of $W(k)$, so $W(k)$ is invertible. Now note that $(\xi\otimes \eta)(\eta\otimes\xi)=\|\eta\|^2 \xi\otimes\xi=\|\eta\|^2p$, thus
\begin{align*}
\|W(k)\eta\|^2&=\langle (\lambda_k(\eta\otimes\xi))^{\dagger}\lambda_k(\eta\otimes\xi)\xi,\xi\rangle=\langle\lambda_k((\eta\otimes\xi)(\xi\otimes\eta))\xi,\xi\rangle\\
&=\|\eta\|^2 \langle \lambda(k)(p)\xi,\xi\rangle=\|\eta\|^2 \langle p\xi,\xi\rangle =\|\eta\|^2\langle \xi,\xi\rangle=\|\eta\|^2,
\end{align*}
showing that $W(k)$ is an isometry, thus it must be unitary. We now claim that $\lambda$ is implemented by $W$:  to prove it we compute
\begin{equation}
W(k)X\eta=\lambda_k( (X\eta)\otimes \xi)\xi=\lambda_k(X \cdot\eta\otimes \xi)\xi=\lambda_k(X)\lambda_k(\eta\otimes\xi)\xi=\lambda_k(X)W(k)\eta,
\end{equation}
thus $W(k)X=\lambda_k(X)W(k)$ and $\lambda_k=W(k)\cdot W(k)^{\dagger}$ as claimed. Then $k^{\mathbb C}=e^{\ad Z(k)}\lambda_k=e^{\ad Z(k)}W(k)\,\cdot\, W(k)^{\dagger}$, and we have proved the first assertion. Now note that if $\|k-j\|<1$ then $\|jk-1\|<1$ and $jk$ can be represented as above, an the second claim follows. For the third claim, consider the map $k\mapsto k(\bullet \otimes \xi)$. We claim that it is real analytic as a map form the Lie group $\Aut$ into the Banach space $\B(\h, \bh)$. Let $k\in \Aut$ with $\|k-1\|<1$; since $\Aut$ is a Banach-Lie subgroup of $\glv$, with Banach-Lie algebra $\der$ (Theorem \ref{autisbanachlie}),  we can use $H\mapsto ke^H$ as a chart of $\Aut$ around $k$, for sufficiently small $H\in \der$. We have
\begin{align*}
&\|ke^H(\bullet \otimes\xi)-\sum_{n=0}^N k\frac{H^n}{n!}(\bullet\otimes \xi)\|=\sup_{\|\eta\|=1} \|e^H(\eta\otimes\xi)-\sum_{n=0}^N \frac{H^n}{n!}(\eta\otimes \xi)\| \\
& \le  \|e^H-\sum_{n=0}^N\frac{H^n}{n!}\|\,\sup_{\|\eta\|=1} \|\eta\otimes\xi\| \le  \|e^H-\sum_{n=0}^N\frac{H^n}{n!}\|.
\end{align*}
Now the last term converges to $0$ as $N\to\infty$, and this computation shows that $\sum_{n=0}^N k\frac{H^n}{n!}(\bullet\otimes \xi)$ is the Taylor polynomial of our map, and it converges uniformly to it, so our map is real analytic. Now $k\mapsto Z(k)$ is real analytic, so is $e^{Z(k)}$ and the product  in the Banach-Lie group $\mathcal U(\h)$, we have that $k\mapsto F(k)=e^{Z(k)}k(\bullet\otimes\xi)e^{-Z(k)}$ is real analytic. It is then apparent that $W(k)=ev_{\xi}(F(k))=F(k)\xi$ is real analytic.
\end{proof}

\begin{rem}If we modify $W$ above with $U(k)=\lambda (k)W(k)$ with a non-continuous function $\lambda:\Aut\to S^1$, we see that it is possible that $t\mapsto k_t=U_t\cdot U_t^{\dagger}$ is a smooth path while $t\mapsto U_t$ is not even continuous.
\end{rem}

\begin{rem}[$\Aut$ as an homogeneous manifold of the unitary group $\mathcal U(\h)$]
Consider the action $A:\mathcal U(\h)\times \Aut\to \Aut$ given by $\U\cdot k=A(U,k)=UkU^{\dagger}=\Ad_U k$. This action is smooth and transitive by Theorem \ref{gomegabh}. Moreover by the same theorem if we fix an orthonormal basis of $\h$, and let $J$ be the antilinear conjugation in that basis, we see that $\Aut$ is the disjoint union of the two open-closed orbits 
$$
\mathcal O(Id)=\mathcal U(\h)\cdot Id=\Aut_0 \quad\textrm{  and } \quad \mathcal O(j)=\mathcal U(\h)\cdot j=\Ad_J\Aut_0,
$$
where $j=\Ad_J=J\cdot J\in \Aut$. That is  $\Aut= \mathcal O(Id)\sqcup \mathcal O(\Ad_J)$. Note that the isotropy group for both orbits is $K=S^1 1$ by Theorem \ref{gomegabh}. It is also apparent from Theorem \ref{levantada} that if we let $s(k)=e^{Z(k)}W(k)$ for an automorphism close to $1$, then 
\end{rem}

\begin{teo}
The action $\mathcal U(\h)\curvearrowright\Aut$ has smooth local cross-sections: for any $k\in \Aut$ there exists an open neighbouhood $V$ of $k$ and a smooth map $s:V\to \mathcal U(\h)$ such that $s(k)=1$, $\Ad_{s(k)}=id_V$.
\end{teo}

Thus in particular the maps $\pi_{Id}: U\mapsto \Ad_U$ and $\pi_{j}: U\mapsto j\Ad_U$ are smooth  open projections, and 
$$
1\rightarrow S^1\rightarrow \mathcal U(\h) \rightarrow \Aut\rightarrow 1
$$
is a smooth principal bundle with structure group $S^1=\mathcal U(1)$ (see \cite[Chapter 3]{schotten} for applications to quantization).

\subsection{Derivations, the structure group and its Lie algebra}

From the fact that $\U_A(B)=ABA$ for all $A,B\in\B$, we derive that $\U_{X,Y}(A)=1/2(XAY+YAX)$, therefore the condition for being in the Lie algebra of the structure group is: $H\in \B(\bh)$ is in $\str$ if and only if there exists $\overline{H}\in \B(\bh)$ such that 
\begin{equation}\label{strbh}
X A H(X)+ H(X)AX= H( XAX)-X\overline{H}(A)X
\end{equation}
for all $A,X\in \bh_{sa}$. On the other hand the condition for $g\in \Str$ in this case can be written as 
$$
g(X)Ag(X)=g(Xg^{-1}( g(1)Ag(1) ) X), 
$$
therefore differentiating $g_t\subset \Str$ at $t=0$, if $g_0=Id$ and $g_0'=H$, we have that 
$$
H(X)AX+XAH(X)=H(XAX)-XH(A)X+XH(1)AX+XAH(1)X,
$$
thus $\overline{H}(A)=H(A) -H(1)A-AH(1)=H-2\U_{H1,1}$ as we mentioned before.

\begin{rem}[Derivations]
If $Z\in \bh$ is skew-adjoint, let $H=\overline{H}=\ad Z$, then $H$ maps $\V=\bh_{sa}$ into itself and it is bounded there. From
$$
XA[Z,X]+[Z,X]AX=XAZX-XAXZ+ZXAZ-XZAX=[Z,XAX]-X[Z,A]X
$$
we have that $H,\overline H\in \bv$ obey equation (\ref{strbh}), thus $H=\ad Z\in \str$ and $\overline{H}=H$. Since $H(1)=[Z,1]=0$, we conclude that $H\in \der$. On the other hand, it was shown in  \cite{semrl} that any complex derivation $\delta$ in $\bh$ must be of the form $X\mapsto XT-TX$ for some bounded linear $T$, thus by complexiying a derivation $D\in \der$ we see that
$$
\der=\{\ad Z: Z\in \bh, Z^{\dagger}=-Z\}. 
$$
Here is a different proof of this equality: take $k_t\subset \Aut$ such that $k_0 = Id$ and $k_0' = D\in\der$. Abusing notation, let $k_t$ denote also the complexification of $k_t$, then $k_0'$ is the complexification of $D$. Now for each $\eta\in\h$, by the previous theorem we can write 
\begin{equation}\label{WX}
W_tX\eta=\lambda_t( (X\eta)\otimes \xi)\xi=\lambda_t(X \cdot\eta\otimes \xi)\xi=\lambda_t(X)\lambda_t(\eta\otimes\xi)\xi=\lambda_t(X)W_t\eta,
\end{equation}
where $\lambda_t=e^{-\ad Z_t}k_t^{\mathbb C}$. Since  $W_t\eta$ is smooth, it defines a skew-adjoint operator $W_0'$ by means of $W_0'\eta=(W_t\eta)'|_{t=0}$ for each $\eta\in \h$. If we differentiate (\ref{WX}) at $t=0$, and we get
\begin{align*}
W_0' X\eta & =\lambda_0'(X) W_0\eta+ \lambda_0(X)W_0'\eta=-\ad Z_0'(X) (\eta\otimes\xi)\xi +k_0'(X)(\eta\otimes\xi)\xi  +XW_0'\eta\\
& =-\ad Z_0'(X)\eta+ D^{\mathbb C}(X)\eta+XW_0'\eta.
\end{align*}
Thus $D^{\mathbb C}(X)=[Z_0',X]+[W_0',X]$, and if we define $Z=Z_0'+W_0'$, we have that $Z^{\dagger}=-Z$, that  $D^{\mathbb C}=\ad Z$ and then $D=\ad Z$ also.
\end{rem}

\begin{rem}[One-parameter groups]
If $U_t=e^{tZ}\subset \mathcal U(\h)$ is a one-parameter group, it is apparent that $k_t=\U_t \cdot \U_t^{\dagger}\subset \Aut$ is also a one-parameter group. The converse holds for our local cross-section: if $k_t=e^{tD}$ is a one-parameter group of automorphims of $\V$, then we now know that $D=\ad Z$ for some skew-adjoint $Z$, and using the formulas we see that $Z_t=Z(k_t)=tZ$. From there  the lift to $\mathsf \U(\h)$ of $k_t$ is simply $s_t=S(k_t)=e^{tZ}$. Rephrasing: if $D=\ad Z\in\der$ then the cross-section gives $s(e^D)=e^Z$.
\end{rem}

\smallskip

A characterization of the Lie algebra of the structure group is also at hand:

\begin{defi}For $T\in\bh$ we denote $\ell_T(X)=TX$ and $r_T(X)=XT$ for $X\in \bh$, that is $\ell$ and $r$ are left and right multiplication in the associative algebra $\bh$.
\end{defi}

\begin{coro} If $\V=\bh_{sa}$ as a JB-algebra, then $\str=\{\ell_T+r_{T^{\dagger}}: T\in \bh\}$. 
\end{coro}
\begin{proof}
Each of the morphisms $H=\ell_T+r_{T^{\dagger}}$ is linear continuous and preserves $\V$, since $TA+AT^{\dagger}$ is self-adjoint for self-adjoint $A$. It is easy to check that if we take $\overline{H} =-(\ell_{T^{\dagger}}+r_T)$ then (\ref{strbh}) is verified, thus $H\in \str$.  Now recall that   $\str=\mathbb L\oplus \der$, thus for  $g\in \str$ we have by the previous remark that $g=L_X+\ad Z$ for some $X\in \V$ and $Z\in \bh$ with  $Z^{\dagger}=-Z$. But then calling $Y=X/2\in \V$, we see that 
$$
g(A)=1/2(XA+AX)+ZA-AZ= (\ell_Y+r_Y+\ell_Z-r_Z)(A)=(\ell_{Y+Z}+ r_{Y^{\dagger}+Z^{\dagger}})(A)
$$
Since $T=Y+Z$ is a generic element of $\bh$, the proof is finished.
\end{proof}

\begin{rem}In particular we obtain  $-Id\in \str$ taking  $T=-1/2$; notice that since $\ell,r$ commute, then
$$
\exp(\ell_T+r_{T^{\dagger}})=e^{\ell_T}e^{r_{T^{\dagger}}}=\ell_{e^T}r_{e^{T^{\dagger}}}, 
$$
that is $e^{\ell_T+r_{T^{\dagger}}}A=e^TAe^{T^{\dagger}}=fAf^{\dagger}$ with linear $f$ (this is apparent because everything happens inside  $\Str_0\subset \GO$ if we exponentiate  $\str$).
\end{rem}

\medskip

From Theorem \ref{stru} of the previous section, and because $p=0,1$ are the only central projections of $\bh$, we conclude this paper by noting that

\begin{teo}If $\V=\bh_{sa}$ as a JB-algebra, then $\Str=\GO\sqcup -\GO$.
\end{teo}

The characterization of the structure group of a product of copies of $\bh_{sa}$ is also at hand. Apparently also, the real part of a connected $C^*$-algebras, when viewed as a JB-algebra, has a  two-components structure group, as the same Theorem \ref{stru} shows. On the other end of the zoo, a $C^*$-algebra with infinitely many central projections shows us that $\Str$ can have infinitely many components.

\section*{Acknowledgments} 

This research was supported by Universidad de Buenos Aires, Agencia Nacional de Promoci\'on de Ciencia y Tecnolog\'\i a (ANPCyT-Argentina) and Consejo Nacional de Investigaciones Cient\'\i ficas y T\'ecnicas (CONICET-Argentina). 
This research was supported by Universidad de Buenos Aires, Agencia Nacional de Promoci\'on de Ciencia y Tecnolog\'\i a (ANPCyT-Argentina) and Consejo Nacional de Investigaciones Cient\'\i ficas y T\'ecnicas (CONICET-Argentina). This line of research on Jordan Banach algebras was encouraged by the talks and informal discussions held by G. Larotonda with Cho-Ho Chu, Bas Lemmens, Jimmy Lawson, Yongdo Lim, Karl-Hermann Neeb and Harald Upmeier among others, during two workshops on Jordan Algebras and Convex Cones (Leiden 2017 and Jeju 2019).

\end{document}